\documentclass[11pt,centertags,oneside]{amsart}
\usepackage{amsmath,amstext,amsthm,amscd,typearea,hyperref}
\usepackage[a4paper,width=16.2cm,top=3cm,bottom=3cm]{geometry}
\usepackage{amssymb}
\usepackage{a4wide}
\usepackage[mathscr]{eucal}
\usepackage{mathrsfs}
\usepackage{typearea}
\usepackage{charter}
\usepackage[english]{babel}
\usepackage[utf8]{inputenc}  
\usepackage[T1]{fontenc}   
\DeclareUnicodeCharacter{00A0}{~}
\usepackage{enumerate}

\usepackage[all]{xy}
\usepackage{hyperref}
\usepackage{ amssymb }

\usepackage{eurosym}
\usepackage{psvectorian}
\usepackage{enumitem}
\usepackage{graphicx}  
\usepackage{xcolor,soul}
\usepackage{tikz-cd}

\newcommand{\R}{\mathbb{R}}  
 
\newcommand{\N}{\mathbb{N}} 
 
\newcommand{\C}{\mathbb{C}} 
\newcommand{\D}{\mathbb{D}}

\newcommand{\la}{\lambda}

\newcommand{\chat}{\mathbb{P}^1}

\newcommand{\ra}{\rightarrow}

\newtheorem{theo}{Theorem}
\newtheorem{prop}[theo]{Proposition}
\newtheorem{coro}[theo]{Corollary}

\newtheorem{defi}[theo]{Definition}
\newtheorem{lem}[theo]{Lemma}

\numberwithin{theo}{section}

\theoremstyle{definition}
\newtheorem{defn}[theo]{Definition}
\newtheorem{rem}[theo]{Remark}

\newcommand{\rs}{\mathbb{P}^1}

\newcommand{\id}{\mathrm{Id}}

\newcommand{\card}{\mathrm{card}\,}
\newcommand{\acal}{\mathcal{A}}
\newcommand{\lam}{\lambda}
\newcommand{\ocal}{\mathcal{O}}

\newcommand{\re}{\mathrm{Re}\,}

\newcommand{\eps}{\epsilon}

\renewcommand{\emptyset}{\varnothing}

\def\C{{\mathbb C}}
\def\N{{\mathbb N}}
\def\R{{\mathbb R}}

\def\H{{\mathbb H}}

\def\B{\mathbb{B}}
\def\D{\mathbb{D}}

\def\id{\mathrm{id}}

\newcommand{\dist}{\mathrm{dist}}

\definecolor{Violet}{cmyk}{0.79,0.88,0,0}
\definecolor{ForestGreen}{cmyk}{0.76, 0, 0.76, 0.45}
\definecolor{Lavender}{cmyk}{0,0.48,0,0}
\newcommand{\violet }{\color{Violet}}

\newcommand{\g}{\color{ForestGreen}}

\renewcommand{\AA}{\mathcal{A}}
\renewcommand{\phi}{\varphi}

\newlength{\wdth}

\makeatletter
\@namedef{subjclassname@2020}{\textup{2020} Mathematics Subject Classification}
\makeatother

\title{Bifurcations for families of Ahlfors island  maps}

\author{Matthieu Astorg$^{\dag}$}
\author{Anna Miriam benini$^*$}
\author{N\'uria Fagella$^\ddag$}

\thanks{$^{\dag}$ Partially supported by the French ANR grant  PADAWAN /ANR-
	21-CE40-0012-01.}
\thanks{$^*$  Partially supported by the French Italian University and Campus France through the Galileo program, under
the project From rational to transcendental: complex dynamics and parameter spaces; by GNAMPA, INdAM; by PRIN Real and Complex Manifolds: Topology, Geometry and holomorphic dynamics n.2017JZ2SW5.}
\thanks{$^\ddag$  Partially supported by grants PID2023-147252NB-I00 and CEX2020-001084-M (Maria de Maeztu Excellence program) from the Spanish state research agency,  and  ICREA Acad\`emia 2020 from the Catalan government.  }
\address{ M. Astorg:  Institut Denis Poisson\\ Université d'Orléans\\
	France} \email{matthieu.astorg@univ-orleans.fr}
\address{ A.M. Benini: Dipartimento  di   Matematica Fisica e Informatica\\
	Universit\'a di Parma\\ Italy
} \email{ambenini@gmail.com}
\address{ N.  Fagella$^{1,2}$:  \newline 
	$1$ Dep. de Matem\`atiques i Inform\`atica\\ Universitat de Barcelona\\ Barcelona\\  Spain \newline
	$2$ Centre de Recerca Matem\`atica\\ Barcelona\\ Spain} \email{nfagella@ub.edu}


\begin{document}

\date{\today}
\setstcolor{red}

\begin{abstract}
	We extend Mañé-Sad-Sullivan and Lyubich's equivalent characterization of stability 
	to the setting of   Ahlfors island maps,  which include notably all meromorphic maps. As a consequence we also obtain the density of $J$-stability for finite type maps in the sense of Epstein. 
\end{abstract}

\maketitle

\section{Introduction}

 A fundamental type of questions in complex dynamics relates to parameter spaces. Consider  a  complex connected manifold $M$ and a  holomorphic family 
$\{f_\lam\}_{\lam \in M}$ of holomorphic self-maps $f_\lam: X \to X$ of a Riemann surface $X$, or more generally, of partially defined maps $f_\lam : W _\lam \to X$, with $W_\lam \subset X$. How are global dynamics of $f_\lam$ affected by a variation of the parameter $\lam$? Of particular interest is the set of parameters $\lam \in M$ for which the global dynamics
is stable under perturbation of the parameter in some sense; this is called the \emph{stability} locus, and its complement is the \emph{bifurcation} locus.  
For families of rational maps on $\rs$ of given degree, much is understood. 
In the seminal works \cite{mss} and \cite{lyu1},
it has been proven that several possible notions of stability coincide:  continuous motion of the Julia set ($J$-stability), 
stability of periodic orbits, and stability of critical orbits (passivity).
As a consequence, the  authors obtained the following crucial result: stability is open and dense 
in any holomorphic family of rational maps.

In \cite{EL}, Eremenko and Lyubich extended this result to the setting of \emph{natural families} of \emph{finite type entire maps}.
An entire map is of finite type if it has only finitely many singular values (see Section \ref{sec:preliminaries}   for the definition of singular values, and see Definition \ref{def:natural} for the definition of a natural family). 
The key point in the proof is the absence of collisions between periodic cycles and the essential singularity $\infty$.  
 In  \cite{ABF}, the analogous result was proven for natural families of  finite type \emph{meromorphic} maps. In this setting, due to the presence of poles, it turns out that collisions between periodic points and  $\infty$ do actually occur, creating a new type of bifurcation to be investigated. 

In this paper, we extend the equivalent characterizations of stability to a much broader class of maps, called Ahlfors Island maps, and we show that stability is dense for a subclass of these maps, namely finite type maps. Both of these classes  were introduced by Epstein in his PhD thesis (\cite{epstein1993towers}), see also \cite{remperippon} and \cite{rempevolker}.

\begin{defi}[Ahlfors island map]
	Let $X$ be a compact Riemann surface and $W \subset X$  a non-empty open set. Let $f: W \to X$ be a holomorphic map.
	We say that $f$ has the $N$-island property if given any $N$ Jordan domains $D_1, \ldots, D_N \subset X$ whose closures are pairwise disjoint and any open set $U$ intersecting $\partial W$, there exists $1 \leq i \leq N$ and an open set $\Omega \Subset U \cap W$ such that $f: \Omega \to D_i$ is a conformal isomorphism.
	
	If there exists $N \geq 1$ such that $f$ has the $N$- island property, we say that $f$ is an Ahlfors island map.
\end{defi}

\begin{defi}[Finite type map]
	Let $X$ be a compact Riemann surface, and let $W \subset X$ be a non-empty open set. Let $f: W \to X$ be a holomorphic map. We say that $f$ is a finite type map on $X$ if 
	\begin{enumerate}
		\item $f$ is non-constant on every connected component of $W$
		\item $f$ has no removable singularities
		\item The set of singular values $S(f)$ is finite.
	\end{enumerate}
\end{defi}

Although it is not immediately apparent from these definitions, finite type maps form a subclass of Ahlfors island maps: Epstein proved that finite type maps have the $N+1$-island property, where $N=\card S(f)$ (see \cite[Proposition 9 p.~88]{epstein1993towers}).
The following are examples of finite type maps:
\begin{itemize}
	\item Rational maps on the Riemann sphere $\chat$ of degree $d \geq 2$.
	\item Entire or meromorphic maps $f: \C \to \chat$ with finitely many singular values.
	\item Maps $f:\rs \setminus E \to \rs$ meromorphic outside of a compact totally disconneted set $E$ (see \cite{BDH,BDH2}) and with finitely many singular values.
	\item Universal covers $f: \D \to \rs \setminus \{0,1,\infty\}$ of the thrice punctured sphere, with  $X=\rs$, $W=\D$; in that case $S(f)=\{0,1,\infty\}$.
\item  Horn maps of rational maps (see \cite{epstein1993towers} for a proof of this fact, and see e.g. \cite{bee} for definitions and properties of horn maps). These maps arise as geometrical limits of rational maps with a parabolic fixed point.
\item Any composition or iterate of finite type maps. 
\end{itemize}

More generally, examples of Ahlfors island maps include: 

\begin{itemize}
	\item Arbitrary meromorphic maps: by a deep theorem of Ahlfors  (see \cite{Ahlfors5Islands}, \cite[Theorem A.2]{Bergweiler5Islands}), any transcendental meromorphic map $f: \C \to \chat$ has the 5-island property, hence is an Ahlfors island map. 
	\item Horn maps of semi-parabolic Hénon maps (introduced in \cite{BSU}), satisfy a property which is very close to the Ahlfors island property, called the small island property (see \cite{AB3}). All of our results on Ahlfors island maps also apply to maps with the small island property, although we have chosen to work within the more established setting of Ahlfors island maps.
	\item Any composition or iterate of Ahlfors island maps is also an Ahlfors island map.
\end{itemize}

 Observe that if $f$ has the $N$-island property, then $f$ can omit at most  $N-1$ points in $X$. 
So for instance if $C \subset \rs$ is a closed infinite set, and $f: W \to \rs \setminus C$  is a covering map with $W \subset \rs$,  then $f$ is not an Ahlfors island map.

\medskip

When working with families of rational maps, it is natural to ask that the degree remains fixed throughout the family. We require a similar type of condition for the families of dynamical systems that we will consider, although they are of course typically of infinite degree.
A somewhat stronger but very convenient notion is that of so-called natural families, to our knowledge first introduced in \cite{EL} in the context of entire maps.

\begin{defi}[Natural families]\label{def:natural}
	Let $f: W \to X$ be a holomorphic map, where $X$ is a compact Riemann surface,  $W \subset X$ is an open set, and let  $M$ be a connected complex manifold. A natural family is a family of holomorphic  maps $\{ f_\lam : W_\lam \to X \}_{\lam \in M}$ of the form 
	\begin{equation}\label{eq:natural relation}
	f_\lam:=\varphi_\lam \circ f \circ \psi_\lam^{-1},
	\end{equation}
	where $\varphi_\lam, \psi_\lam: X \to X$ are quasiconformal homeomorphisms depending holomorphically on $\lam \in M$, and $W_\lam=\psi_\lam(W)$.
\end{defi}

Let  $f_\lam:=\varphi_\lam \circ f \circ \psi_\lam^{-1}$ be as above. The following basic but important facts follow directly from the fact that $\varphi_\lambda,\psi_\lambda$ are homeomorphisms:
\begin{enumerate}
	\item[\rm (a)] $\varphi_\lam$ maps $S(f)$ to $S(f_\lam)$
	\item[\rm (b)]  $\psi_\lam$ maps critical points of $f$ to critical points of $f_\lam$, preserving multiplicities
	\item[\rm (c)] if $f: W \to X$ is a finite type map (respectively an Ahlfors island map), then so are the maps 
	$f_\lam : W_\lam \to X$.
\end{enumerate}

\begin{rem}\label{rem:marking}
We can think  of $\varphi_\la$ and $\psi_\la$ as a {\em marking}, giving a motion of the singular values and the critical points respectively. In many cases, we can modify $\phi_\la$  without changing the motion of the $S(f_\la)$. In particular, for any $\la_0 \in M$, we can always assume without loss of generality that $\varphi_{\la_0}=\psi_{\la_0}={\rm Id}$, by redefining $\tilde{\psi}_\la=\psi_\la \circ \psi_{\la_0}^{-1}$ and $\tilde{\varphi}_\la=\varphi_\la \circ \varphi_{\la_0}^{-1}$. We refer to this as \emph{choosing $\la_0$ as basepoint.}
\end{rem}

Generalizing \cite{EL} and \cite{gold-keen}, Epstein proved in \cite{epstein1993towers} that one can always 
embed any finite type map in a (locally) natural family of dimension $\card S(f)$. 
Moreover, this family satisfies a universal property in the sense that any other natural family can be lifted to it. More recently, these results were partially extended to the case of arbitrary entire maps by Ferreira and van Strien (\cite{ferreiravs}).
We will however not require any of these results, and we refer the reader to \cite{epstein1993towers},
\cite{epsteintransversality}, and \cite{ferreiravs}.

Finally, we recall the notion of activity or passivity of a singular value (see \cite{Levin}, \cite{mcmbook}), adapted to our context.

\begin{defi}[Passive singular value]\label{defn:active singular values}
	Let $\{f_\lam = \phi_\la \circ f_{\la_0} \circ \psi_\la^{-1}\}_{\lam \in M}$ be a natural family of Ahlfors island 
	maps. Let $v(\lam):=\varphi_\la(v)$ be a singular value of $f_\lam$ 
	depending holomorphically on $\lam$ near some $\lam_0 \in M$.
	We say that $v(\lam)$ is \emph{passive} at $\lam_0$ if there exists a neighborhood $V$ of $\lam_0$ in $M$ such that:
	\begin{enumerate}
		\item	either $f_\lam^n(v(\lam))\in X \setminus W_\lambda$  for some $n\in\N$ and for all $\lam \in V$; or
		\item the family
		$\{\lam \mapsto f_\lam^n(v(\lam)) \}_{n \in \N}$ is well-defined and normal on $V$.
	\end{enumerate}
We say that  $v(\la)$ is {\em active} at $\la_0$ if it is not passive. 
\end{defi}

 In  classical terminology, a natural family is $J$-stable if there is a holomorphic motion of the Julia set respecting the dynamics  (see Definition \ref{def:holomotion} for a definition of a holomorphic motion).
In the setting of general Ahlfors island maps, we need to impose a compatibility condition between this holomorphic motion and the one induced by $\phi_\la$ on the singular value set $S(f)$.

\begin{defn}[$J$-stability]\label{defi:jstab}
	Let $\{f_\la = \phi_\la \circ f_{\la_0} \circ \psi_\la^{-1}\}_{\la \in U}$ be a natural family of Ahlfors island maps with basepoint $\la_0 \in U$. We  say that $\{f_\la\}_{\la \in U}$ is $J$-stable if there exists a holomorphic motion 
	$H: U \times J(f) \to X$ with basepoint $\la_0$ such that 
	\begin{enumerate}
		\item\label{i:item1} for all $\la \in U$, $h_\la:= H(\la, \cdot): J(f_{\la_0}) \to J(f_\la)$ and 
		$h_\la \circ f_{\la_0} = f_\la \circ h_\la$
		\item\label{i:item2} for all $\la \in U$, $h_\la = \phi_\la$ on $S(f_{\la_0}) \cap J(f_{\la_0})$.
	\end{enumerate}
\end{defn}

In the classical setting of rational maps, $J$-stability is defined by the sole condition (\ref{i:item1}) (see \cite{mss} and \cite{lyu1}).
This remains true more generally for finite type maps, and more subtly, for meromorphic maps whose singular value set has empty interior, as stated in the following proposition.

\begin{prop} \label{prop:equivalent_definitions}
	Let $\{f_\la = \phi_\la \circ f_{\la_0} \circ \psi_\la^{-1}\}_{\la _in M}$ be a natural family of Ahlfors island  maps satisfying \eqref{i:item1} in Definition \ref{defi:jstab}. Assume that either 
	\begin{enumerate}
		\item the map $f_{\la_0}$ is a finite type map; or
		\item the map $f_{\la_0}$ is meromorphic on $\C$ and $S(f_{\la_0})$ has no interior.
	\end{enumerate}
	Then $\{f_\la = \phi_\la \circ f_{\la_0} \circ \psi_\la^{-1}\}_{\la \in M}$ also satisfies \eqref{i:item2} in Definition \ref{defi:jstab}.
\end{prop}

In general, given a natural family of Ahlfors island maps $\{f_\la = \phi_\la \circ f_{\la_0} \circ \psi_\la^{-1} \}_{\la \in M}$, 
the maps $\phi_\la$ and $\psi_\la$ are not unique, that is, there can exist quasiconformal self-homeomorphisms of $X$,  $\tilde \phi_\la, \tilde \psi_\la$ such that for all $\la \in M$, 
$$f_\la = \phi_\la \circ f_{\la_0} \circ \psi_{\la}^{-1} = \tilde \phi_\la \circ f_{\la_0} \circ \tilde \psi_{\la}^{-1}.$$
It is natural to ask whether our Definition \ref{defi:jstab} of $J$-stability depends on the choice of $\phi_\la$ and $\psi_\la$ or only on $f_\la$.
In the case where $\partial W_\la$ is totally disconnected (which includes in particular the case of meromorphic maps and their iterates), the answer is negative.

\begin{prop} \label{prop:indepe}
	Let $f_{\la_0}$ be an Ahlfors island map such that $\partial W(f_{\la_0})$ is totally disconnected, and assume that there are quasiconformal homeomorphisms of $\rs$,  $\phi_\la, \psi_\la, \tilde \phi_\la$ and $\tilde \psi_\la$ depending holomorphically on $\la \in M$ such that for all $\la \in M$,
	$$f_\la := \phi_\la \circ f_{\la_0} \circ \psi_{\la}^{-1} = \tilde \phi_\la \circ f_{\la_0} \circ \tilde \psi_{\la}^{-1}.$$
	Then $\phi_\la = \tilde \phi_\la$ on $S(f_{\la_0})$.
\end{prop}

The proof uses the Lehto-Virtanen theorem; the assumption that $\partial W(f_{\la_0})$ is totally disconnected will be used to ensure that all asymptotic 
paths land in $X$.

We may now state our main results.

\begin{theo}[$J-$stability of Ahlfors Islands maps]\label{th:mssahlfors}
	Let $\{f_\lam\}_{\lam \in M}$ be a natural family of Ahlfors island maps  which are not automorphisms of $X$.
	Let $\lam_0 \in M$. Then the following are equivalent:
	\begin{enumerate}
		\item there exists a neighborhood of $\lam_0$ on which all singular values are passive.
		\item the family is $J$-stable on a neighborhood of $\lam_0$.
	\end{enumerate}
\end{theo}

As in \cite{ABF}, one of the main difficulties in the proof of this type of result 
is the possibility of collisions between periodic cycles and the boundary $\partial W_\la$ of the domain of definition $W_\la$ of the map $f_\la$. In \cite{ABF}, a delicate analysis allowed us to relate this phenomenon to the activity of asymptotic values (see in particular Theorem A). However, this analysis used in a crucial way the fact that $\partial W_\la=\{\infty\}$ in the case of a family of meromorphic maps.
In the more general setting of Ahlfors island maps or even finite type maps, $\partial W_\la$ may contain a continuum, and tracts above asymptotic values may accumulate on such continua in a complicated way. This is why instead of working with the motion of periodic cycles, we consider the motion of the backward orbit of some point $z \in J(f_{\la_0})$, and relate this holomorphic motion (or lack thereof) to the activity or passivity of singular values (see Proposition \ref{prop:motiongo}).

\medskip

In view of Theorem \ref{th:mssahlfors},
we hereafter define the {\em stability locus} of a natural family $\{f_\la \}_{\la \in M}$ as the set of all $\la \in M$ for which one of the above equivalent conditions is satisfied, and the {\em bifurcation locus} as its complement, extending these classical notions to the very general setting of Ahlfors island maps.

In the case of finite type maps, we can obtain one more equivalent characterization of stability.

\begin{theo}[$J-$stability of finite type maps]\label{th:MSS}
	Let $\{f_\lam\}_{\lam \in M}$ be a natural family of finite type maps which are not automorphisms of $X$, and let $\lam_0 \in M$.
	The following are equivalent:
	\begin{enumerate}
		\item there exists a neighborhood of $\lam_0$ on which all singular values are passive
		\item the family is $J$-stable on a neighborhood of $\lam_0$ 
		\item there is a constant $N \in \N$ and a neighborhood $U \subset M$ of $\lam_0$ such that for all $\lam \in U$, 
		the period of any attracting cycle is at most $N$.
	\end{enumerate}
\end{theo}

As a consequence, we obtain (as in the case of rational or finite type  meromorphic maps):

\begin{coro}[Density of $J-$stable maps] \label{coro:stabdense}
	The stability locus is open and dense in natural families of finite type maps.
\end{coro}

On the other hand, using Theorem \ref{th:mssahlfors} is is very easy to construct families of Ahlfors island maps 
with robust bifurcations, i.e. for which the set of $J$-stability is not dense. For instance:


\begin{coro}[ Bifurcation locus with nonempty interior]\label{coro:robustbif}
	Let $\{f_\la \}_{\la \in M}$ be a natural family of Ahlfors island maps and let $\la_0 \in M$. Assume that
	\begin{enumerate}
		\item $v_{\la_0}$ belongs to the interior of $S(f_{\la_0})$
		\item and $v_{\la_0}$ is active at $\la_0$.
	\end{enumerate}
	Then $\la_0$ is in the closure of the interior of the bifurcation locus.
\end{coro}

There are many examples of Ahlfors island maps $f_{\la_0}$ satisfying the first condition of 
Corollary \ref{coro:robustbif}; in fact, there exists entire or meromorphic maps $f$ with $S(f) = \rs$.

\subsection*{Acknowledgements}
We wish to thank Adam Epstein and Lasse Rempe for helpful discussions. We would also like to thank Lasse Rempe for pointing out the need for item (2) in Definition \ref{defi:jstab} in the setting of Ahlfors island maps.

 \section{Activity Locus of Ahlfors Island Maps}

\subsection{Preliminaries: Ahlfors island maps, finite type maps and singular values}\label{sec:preliminaries}

 Choose an arbitrary hermitian metric on $X$	(the choice is not important since $X$ is compact), and denote by  $d_X$ the distance it induces. Unless otherwise stated, distances in $X$ will be measured in the sense of $d_X$.

We start by recalling the definitions of critical, asymptotic and singular values. Let $W,X$ be two Riemann surfaces and let $f: W\ra X$ holomorphic. A point $c\in W$ is  \emph{critical} if   $f'(c)=0$. 
A  value $v\in X$ is \emph{critical} if it is the image of a critical point; it is \emph{asymptotic} if there is a curve $\gamma: \R_+  \to W$ such that, as $t\ra\infty$,  $\gamma(t)\ra\partial W$   and $f(\gamma(t))\ra v$. (Notice that if $W \subset X$,  we do not require $\gamma(t)$ to converge in $X$). A \emph{logarithmic tract}   over an asymptotic value $v$ is a simply connected domain $T\subset W$ such that $f:T\ra D\setminus\{v\}$ is an infinite degree unbranched covering over a  topological disk   $D$  punctured at $v$.
Any limit point of a (possibly constant) sequence of critical or asymptotic values is called a {\em singular value}.
If we denote by $S(f)$ the set of all singular values of a holomorphic map $f: W \to X$, it is 
a classical result that $f: W \setminus f^{-1}(S(f)) \to X \setminus S(f)$ is a covering map
(assuming that $X \setminus S(f)$ is not empty).

As a consequence, any asymptotic value that is isolated in $S(f)$ admits logarithmic tracts.
This is in particular the case for finite type maps, although not in general for Ahlfors island maps.

\subsection{Ahlfors island maps}

We now move on to basic properties of Ahlfors island maps, including a classification of so-called exceptional Ahlfors island maps which will sometimes require separate arguments. Roughly speaking, these exceptional cases correspond to those for which few or no points can escape the domain of definition, and include notably rational and entire maps.

By considering a sequence of neighborhoods shrinking to a boundary point $z$,  the  island property implies that if $\partial W \neq \emptyset$, then there are at most finitely many values which do not have infinitely many preimages   in $W$ under $f$, and every  such value is a singular value.  In particular, if $f$ is an Ahlfors island map of finite degree, then we must have $W=X$.


\begin{defi}[Fatou and Julia sets]
Let $f : W \to X$ be an Ahlfors island map. The Fatou set $F(f)$ of $f$ is defined as the union of all open subsets $U \subset W$ such that
\begin{enumerate}
\item either there exists $n \in \N^*$ such that $f^n(U) \cap W = \emptyset$, or
\item  $f^n(U)\subset W $ for all $n\in \N$ and $\{f^n :U \to X \}_{n\in\N}$ is normal.
\end{enumerate}
The Julia set is $J(f) := X \setminus F(f)$.
\end{defi}
Observe that by this definition, we have $\partial W \subset J(f)$, where $\partial W$ denotes the boundary of
$W$ as a subset of $X$. Hence if $W \neq X$, then the Julia set is non-empty. 
The theorem below (stated for finite type maps, but whose proof also works for Ahlfors island map) gives a much stronger statement:

\begin{theo}[{\cite[p.~100]{epstein1993towers}}]\label{th:repdense}
	Let $f: W \to X$ be an {Ahlfors island map} which is not an automorphism of $X$.
	Then $J(f)$ is non-empty, perfect, and repelling cycles are dense in $J(f)$.
\end{theo}

\begin{defi}[Exceptional  Ahlfors islands maps]
	Let $f: W \to X$ be  an Ahlfors island map. Let  $W_\infty:=\mathrm{int} \bigcap_{n \geq 0} f^{-n}(W)$.
	We say that $f$ is exceptional if either $f$ is an automorphism of $X$ or $W_\infty$ is non-empty and  non-hyperbolic.
\end{defi}

\begin{lem}[Exceptional Ahlfors islands maps]\label{lem:exceptional}
	An Ahlfors island map  $f:W\ra X$ is exceptional if and only if it is analytically conjugated to one of the following types:
	\begin{enumerate}
		\item a rational self-map of $\rs$
		\item an affine endomorphism of a complex torus
		\item an entire map 
		\item  a meromorphic map with exactly one pole, which is also an omitted value
		\item a self-map of $\C^*$ with essential singularities at $0$ and $\infty$
		\item an automorphism of $X$.
	\end{enumerate}
\end{lem}

\begin{proof}
	 By definition if $f$ is exceptional and not an automorphism then there is at least one {connected component of $W_\infty$} which is not hyperbolic. In particular, there is one connected component of $W$ which is not hyperbolic. Since $W\subset X$, $X$ is also not hyperbolic. 	This means that $W$ and $X$ are  isomorphic to one of the following: $\rs$, $\C$, $\C^*$, or a complex torus.	
	 
	 If $W \simeq \rs$,  then $X=\rs$ and $f$ is a rational map.
	 
If $W$ is a complex torus, then $W=X$ and $f$ must be an affine endomorphism of $X$ (since endomorphisms of complex tori are all affine).
	
	If $W \simeq \C$, then  $X=\rs$ (since $X$ is compact) and $f$ is a transcendental meromorphic map (since the island property implies infinite degree). Moreover, we must have 
	$$\card \bigcup_{n \geq 0} f^{-n}(\{\infty\}) \leq 2,$$ 
	since otherwise $W$ would be hyperbolic. This  implies that $f$ has at most one pole, and  that this pole must be an omitted value; otherwise, by Picard's theorem, we would have  $\card \bigcup_{n \geq 0} f^{-n}(\{\infty\}) = \infty$.

	Finally, if $W \simeq \C^*$, then  the island property implies that $f$ has two essential singularities at $0$ and $\infty$, and  as before we must have $\card \bigcup_{n \geq 0} f^{-n}(\{0,\infty\}) \leq 2$, so both $0$ and $\infty$ are omitted values and $f$ is a self-map of $\C^*$.

	Conversely, it is clear that maps of the form (1)-(6) are exceptional.
\end{proof}

The following is an immediate consequence of the exhaustive list given above.

\begin{coro}\label{cor:hyperbolic are non exceptional}
	If $W$ is hyperbolic and non-compact,  then $f$ is not exceptional.
\end{coro}

Note that the converse is not true, since a meromorphic map with infinitely many poles is not exceptional.

\medskip

In natural families, either the whole family is exceptional  or exceptional maps form a proper analytic subset.
\begin{prop}[Natural family with exceptional maps]\label{prop:sparsexcep}
Let $\{f_\lambda\}_{\lambda\in M}$ be a natural family of Ahlfors island maps. Then either all maps in $M$ are exceptional, or the set of exceptional maps in 
$M$ forms a (possibly empty) proper analytic subset of $M$.
\end{prop}

\begin{proof}
	Let $f_\lam  = \varphi_\lam \circ f_{\lam_0} \circ \psi_\lam^{-1}$	and assume that $f:=f_{\lam_0}$ is exceptional. As usual, we may assume that $\varphi_{\lam_0} = \psi_{\lam_0}=\id$.
If $f$ is an automorphism, so is $f_\lambda$ for all $\lambda\in M$ and we are done. Otherwise,  
	 by Lemma \ref{lem:exceptional},
	 either $X=W$ is a complex torus, or $X=\rs$ and $W=\rs$, $\C$ or $\C^*$.
		 If $W=\rs$ or if $W$ is a complex torus, then clearly all maps in the family 
	are exceptional, so we can further reduce to the case where $W=\C$ or $\C^*$, i.e. $f$ is either meromorphic or 
	defined on $\C^*$ with two essential singularities. 
	
	The meromorphic case was treated in \cite[Prop. 5.4]{ABF}, so we only need to deal with the case where $W=\C^*$, which is similar. By the classification of Lemma \ref{lem:exceptional}, $f$ omits $0$ and $\infty$.	
	Without loss of generality, we may normalize $\psi_\lam$ so that it always fixes $0$ and $\infty$. 
	Then for all $\lam \in M$, $\varphi_\lam(0)$ and $\varphi_\lam(\infty)$ are omitted values of $f_\la$. By Picard's theorem, $f_\lam$ is then exceptional if and only if $\varphi_\lam( \{0,\infty\}) = \{0,\infty\}$.
	Indeed, if this relation is not satisfied, then at least one of $0$ or $\infty$ is not a Picard exceptional value, 
	so that say $\card f_\lam^{-1}(\{0,\infty\})=\infty$ and $f_\lam$ is not exceptional.
	By connectivity of $M$, the set of $\lam \in M$ such that $f_\la$ is exceptional is then exactly
	$$E:=\{ \lam \in M : \varphi_\lam(0)=0 \text{ and } { \varphi}_\lam(\infty)=\infty\}.$$
	This set is either all of $M$ or a (possibly empty) proper analytic subset of $M$.
\end{proof}

We record here the following well-known version of Montel's theorem (see e.g. \cite[Corollary 3.3]{milnor}).

\begin{lem}[Montel's Theorem]\label{lem:normalf}
	Let $U,V$  be hyperbolic Riemann surfaces. Then the family of maps $\{f\mid f: U \to V \text{ holomorphic} \}$ is a normal family.
\end{lem}
. 
Lemma \ref{lem:boundary} below generalizes to the case of non-exceptional Ahlfors island maps the well-known characterization of the Julia set as  the closure of the set of prepoles of $f$. Recall that   a transcendental meromorphic map $f$ is non-exceptional if and only if there is at least one non-omitted pole.

\begin{lem}[Characterization of $J(f)$]\label{lem:boundary}
	Let $f: W \to X$ be a non-exceptional   Ahlfors island map. Then
	$J(f)=\overline{\bigcup_{n \geq 0} f^{-n}(\partial W)}$.
\end{lem}

\begin{proof}	
	The inclusion $\overline{\bigcup_{n \geq 0} f^{-n}(\partial W)} \subset J(f)$ is always true by definition. Conversely, if $z\notin \overline{\bigcup_{n \geq 0} f^{-n}(\partial W)}$, then there exists a neighborhood $U$ of $z$ such that $f^n(U)\cap \partial W=\emptyset$ for all $n\in \N$. Therefore either there exists $n\in \N$ such that $f^n(U)\cap W=\emptyset$, and hence $z\in F(f)$; or $f^n(U)\subset W$  for all $n\in \N$, i.e. $U\subset W_\infty$. Since $f$ is non-exceptional, $W_\infty$ is hyperbolic. Thus  $(f^n|_{W_\infty})_{n\in\N}$
	is normal  by Lemma \ref{lem:normalf} and $z\in F(f)$.

\end{proof}

For a holomorphic map $f$ with an isolated essential singularity,   a point $z_0$ is called Picard exceptional if and only if $f^{-1}(z_0)$ has finite cardinality. 
By analogy, with introduce the following terminology.

\begin{defi}[Picard exceptional value]
	Let $f: W \to X$ be an Ahlfors island map with $\partial W \neq \emptyset$. We will say that $v \in X$ is a Picard exceptional value if 
	$f^{-1}(v)$ is finite (possibly empty).
\end{defi}

By definition, if $f$ has the $N$ island property and $\partial W \neq \emptyset$, then it has at most $N-1$ Picard exceptional values. 
Lemma \ref{lem:densepreimages} below generalizes the characterization of the Julia set as  the closure of the backward orbit of any point $z_0$ which is not Picard exceptional. One can show that Picard exceptional values are always  asymptotic values.   

\begin{lem}[Preimages of non exceptional values are dense in $J(f)$]\label{lem:densepreimages}
	Let $f: W \to X$ be  a non-exceptional Ahlfors island map, and let $p \in \N^*$.
	Let $z_0 \in J(f)$ and assume that $z_0$ is not a Picard exceptional value. Then
	$\bigcup_{n \geq 0} f^{-np}(\{z_0\})$ is dense in $J(f)$. 
\end{lem}

\begin{proof}
	 Since repelling cycles are dense in the Julia set by Theorem \ref{th:repdense}  we have that $J(f)=J(f^p)$, hence  we may assume without loss of generality that $p=1$.	Let $U \subset W$ be an open set which intersects $J(f)$. By Lemma \ref{lem:boundary}, there 
	exists $n \in \N^*$ and $z \in U$ such that $f^n(z) \in \partial W$. Since $z_0$ is not   a Picard exceptional value and 
	$f$ is non-exceptional (hence has infinite degree), $z_0$ has infinitely many preimages : choose $N+1$ such preimages $x_1, \ldots, x_{N+1}$, 
	where $f$ has the $N$-island property, and let $D_1, \ldots, D_{N+1}$ be Jordan domains with pairwise disjoint closures containing each $x_i$. 
	By the island property, there exists $1 \leq i \leq N+1$ and $\Omega \Subset f^n(U) \cap W$ such that 
	$f: \Omega \to D_i$ is a conformal isomorphism. In particular,  $f^{n+2}: U \cap f^{-n}(\Omega) \to f(D_i)$ is well-defined and surjective; in other words, there exists $z_1 \in U$ such that $f^{n+2}(z_1) = z_0$, as desired.
\end{proof}


\subsection{Activity of singular values and preliminary results}

In this section we collect some results about activity and passivity 
of singular values in natural families of Ahlfors island maps.
	In what follows, we fix a natural family $\{f_\lam=\varphi_\la \circ f \circ \psi_\la^{-1}\}_{\lam \in M}$ of Ahlfors island maps.
	
Given a  singular value $v_{\la_0}$ we consider the holomorphic function $v(\la):=\varphi_\la(v_{\la_0})$. Since $\{f_\la\}_{\la\in M}$ is a natural family, $v_\la$ is a singular value for $f_\la$ for each $\la\in M$, of the same type as $v_{\la_0}$. With this in mind  we often refer to $v(\la)$ as a singular value, although technically it is a holomorphic function whose value $v(\lambda)$ is a singular value for  $f_\la$ for each fixed $\lambda$. We also use the equivalent  notation $v_\la=v(\la)$.

Recall the definition of activity/passivity given in the introduction.

\begin{defi}[Passive singular value]
	Let $\{f_\lam\}_{\lam \in M}$ be a natural family of finite type 
	maps. Let $v(\lam)$ be a singular value (or a critical point) of $f_\lam$ 
	depending holomorphically on $\lam$ near some $\lam_0 \in M$.
	We say that $v(\lam)$ is \emph{passive} at $\lam_0$ if there exists a neighborhood $V$ of $\lam_0$ in $M$ such that:
	\begin{enumerate}
		\item	either $f_\lam^n(v(\lam))\in X \setminus W_\lambda$  for some $n\in\N$ and for all $\lam \in V$; or
		\item the family
		$\{\lam \mapsto f_\lam^n(v(\lam)) \}_{n \in \N}$ is well-defined and normal on $V$.
	\end{enumerate}
	We say that  $v(\la)$ is {\em active} at $\la_0$ if it is not passive. 
\end{defi}

\begin{defi}[Activity locus]
	Given  a singular value $v_\la$ we define its  {\em activity locus}  as the set of parameters
	\[
	\mathcal{A}(v_\la) = \{ \la_0 \in M\mid v_\la \text{\ is active at $\la_0$} \}.%
	\]
\end{defi}

 It is important to remark that the concept of activity must be associated only to non-persistent behaviour.

\begin{defi}[Persistence] \label{def:np}
 We say that $f_{\la_0}^n(v(\la_0))\in X \setminus W_{ \lambda_0}$ (resp.   $f_{\la_0}^n(v(\la_0))\in \partial  W_{ \lambda_0}$) \emph{persistently} if for all $\lam$ in a neighborhood of $\lam_0$, we have $f_{\la}^n(v(\la))\in X \setminus W_{\lambda}$  (resp.   $f_{\la}^n(v(\la))\in \partial  W_{\g \lambda}$).
\end{defi}

\begin{lem}[Persistence property]\label{lem:Gct}
	Let $\{f_\lam\}_{\lam \in M}$ be a natural family of finite type 
	maps. Let $v(\lam)$ be a singular value (or a critical point) of $f_\lam$ 
	depending holomorphically on $\lam\in M$ If $v_\la$. If  $n \in \N$ is such that $f_\la^n(v_\la) \in \partial W_\la$ 
	for all $\la$ in an open subset of $M$, then  $f_\la^n(v_\la) \in \partial W_\la$ persistently on $M$.
\end{lem}

 The proof follows almost immediately from the following lemma, which will also be useful later on.
\begin{lem}[{\cite[Lemma 2.11]{ABF}}] \label{lem:qr}
	Let $(\psi_\lam :X\to X)_{\lam \in M}$ be a holomorphic family of quasiconformal homeomorphisms,
	such that $\psi_{\lam_0}=\id$ and $\dim M=1$. Let $g: M \to X$ be a holomorphic map and  $G(\lam):=\psi_\lam^{-1} \circ g(\lam)$. Then either $G$ is constant, or there are neighborhoods $U$ of $\lam_0$ in $M$ and $V$ of $G(\lam_0)$ in $X$ such that $G: U \to V$ is a branched cover, ramified only possibly at $G(\lam_0)$.
\end{lem}

In fact, one could prove that $G$ is even quasiregular (\cite{bert}), although we will not require this.

\begin{proof}[Proof of Lemma \ref{lem:Gct}]
	Let $G(\la):=\psi_\la^{-1} \circ f_\la^n(v_\la)$. By the previous Lemma, 
	the map $G: M \to X$ is either locally constant (hence constant since $M$ is connected) or open. Since $G(\la) \in \partial W$ if and only if $f_\la^n(v_\la) \in \partial W_\la$,  and $\partial W$ has empty interior, the map $G$ cannot be open, and is therefore constant.
\end{proof}

The next lemma, though technical, is standard in rational dynamics. It will be used to locally follow holomorphically preimages of a marked point, 
up to passing to a finite branched cover in parameter space.  Since we work with non-rational maps, we give here a detailed proof.

\begin{lem}[Holomorphic dependence of preimages] \label{lem:marked}
	Let $S,X$ be Riemann surfaces, $U \subset X$ be a domain, and $F: S \times U \to X$ be a
	holomorphic map such that for all $\lam \in S$, the map $F(\lam, \cdot)$ is non-constant on $U$. Let $\gamma: S \to X$ be a holomorphic map, and let $(\lam_0,z_i) \in S \times U$ (where $z_i, 1 \leq i \leq N$ are $N$ distinct points in $U$) be such that $F(\lam_0,z_i)=\gamma(\lam_0)$. Then there is a neighborhood $V$ of $\lam_0$ in $S$, a finite branched cover $\pi: \D \to V$ and  holomorphic maps $\D \ni t \mapsto x_i(t) \in U$ such that $x_i(0)=z_i$ and
	for all $t \in \D$,
	$$F(\pi(t), x_i(t)) = \gamma \circ \pi(t).$$
\end{lem}

\begin{proof}
	Up to restricting $U$, we may assume without loss of generality that $F(\lam_0,\cdot)$ extends holomorphically in a neighborhood of $\overline{U}$ in $X$, and that for all $z \in \partial U$, $F(\la_0,z) \neq \gamma(\la_0)$.
	
	Let 
	$$
	Z:=\{(\lam,(y_1,\ldots, y_N)) \in S \times U^N: F(\lam,y_i)-\gamma(\lam)=0  \text{ for $1 \leq i \leq N$} \}
	$$
	and let $Z_0$ denote the irreducible component of $Z$ containing
	$(\lam_0, z_1, \ldots, z_N)$.
	Let $\pi_S:  (\lam, x_1, \ldots, x_N) \mapsto \lam$ denote the projection from $Z_0$ to $S$.
	 Then $Z_0$ is an analytic subset of $S \times U^N$ of complex dimension one. 
	Indeed, if  $Z_0$ had higher dimension, then $\pi_S^{-1}(\{\la_0\}) \cap Z_0$ would have positive dimension, which would contradict the assumption that $F(\la_0, \cdot)$ is non-constant on $U$.

	 The set of singular points of $Z_0$ is discrete (since it is a codimension at least 1 analytic subset of $Z_0$), and so is the set of critical points of the projection $\pi_S$ restricted to $Z_0$.
	Therefore, there exists a small disk $V \subset S$ containing $\lam_0$ such that 
	$Z_1^*:=Z_0 \cap \pi_S^{-1}(V \setminus \{\lam_0\})$ is smooth and $\pi_S: Z_1^* \to V \setminus \{\lam_0\}$ has no critical points.
	Up to taking a smaller $V$, we may also assume that for all $(\la,z) \in V \times \partial U$, $F(\la,z) \neq \gamma(\la)$. Then  the map ${ \pi_S}: Z_1^* \to V \setminus \{\lam_0\}$ is proper. 
	Indeed, let $K \subset V \setminus \{\la_0\}$ denote a compact set, and let $(\la_n, z_n) \in \pi^{-1}(K)$.
	Up to extracting a subsequence, we may assume that $(\la_n, z_n) \to (\la,z) \in (K \times \overline{U}) \cap Z_1^*$. But by our restriction on $V$, we have $z \in U$. This proves that $\pi_S^{-1}(K)$ is compact in $Z_1^*$, hence that $\pi_S: Z_1^* \to V \setminus \{\la_0\}$ is proper.	
	It is also surjective (because it is open and closed, and $S$ is connected).
	
	Since $\pi_S: Z_1^* \to V \setminus \{\lam_0\}$ is proper, surjective and has no critical points,  it is a covering map. Therefore, there exists a conformal isomorphism $h : \D^*  \to Z_1^*$ and a 
	degree $d \geq 1$ covering map $\pi: \D^* \to V \setminus \{\lam_0\}$ such that 
	$\pi_S \circ  h = \pi$. The map $h$ extends to a holomorphic map $h: \D \to S \times U^N$ 
	such that $h(0) = (\lam_0, z_1, \ldots, z_N)$ and the map $\pi$ extends to a holomorphic map $\pi: \D \to V$ with $\pi(0)=\lam_0$.
	For $1 \leq i \leq N$, let $\pi_i: S \times U^N \to U$ be the projection defined by 
	$\pi_i(\lam, x_1, \ldots, x_N) = x_i$.

	We can then set $x_i:=\pi_i \circ h$; it is straightforward to check that they have the desired properties.
\end{proof}

We now show that if a  singular value is mapped to $\partial  W_{\lambda_0}$ at a parameter $\lambda_0\in M$, the latter parameter can be perturbed in such a way that the singular value has  any prescribed behavior.
\begin{prop}[Shooting Lemma]\label{lem:shooting}
	Let $\{f_\la\}_{\la\in M}$ be a natural family of Ahlfors island maps.
	Let $\la_0\in M$ and $n\geq 0$ be such that a singular value $v_\la$ satisfies $f_{\lam_0}^n(v_{\la_0}) \in \partial W_{\lam_0}$ non persistently. 
	Let $\lambda \mapsto \gamma(\la)$ be a  holomorphic map  such that $\gamma(\la_0)$ is not Picard exceptional. Then we can find $\la'$ arbitrarily close to $\la_0$  such that
	$$f_{\lam'}^{n+2}(v_{\lam'}) = \gamma(\lam').$$
\end{prop}

Since Picard exceptional values are also asymptotic values, Proposition \ref{lem:shooting} applies in particular whenever $\gamma(\la_0) \notin S(f_{\la_0})$.
To prove Proposition~\ref{lem:shooting} we need the following lemma.

\begin{lem}[\cite{ABF}]\label{lem:fp} Let $V$ be a Jordan domain, and let $f,g$ be holomorphic functions in a neighborhood of $\overline{V}$. Suppose that  $g(\overline{V})\subset f(V)$ and  $g(\partial V)\cap f(\partial V)=\emptyset$. Then there exists $\lambda\in V$ such that $f(\lambda)=g(\lambda)$.
\end{lem}

\begin{proof}[Proof of Proposition~\ref{lem:shooting}]

	First, we pick an arbitrary one-dimensional slice containing $\lam_0$ in the parameter space $M$ on which $\lam \mapsto f_\lam^n(v(\lam))$ is not constant, and we identify $M$ with $\D(\lam_0,1) \subset \C$ in the rest of the proof.
	By assumption $f_\la=\varphi_\la\circ f\circ \psi_\la^{-1}$ and we may assume without loss of generality that $\varphi_{\la_0}=\psi_{\la_0}={\rm Id}$ and hence $f=f_{\la_0}$. Let $x:=f_{\lam_0}^n(v_{\lam_0})$, and hence, by assumption, 
	$x \in \partial W$.

	Let $N \in \N$ be such that $f$ (and therefore each map $f_\lam$) have the $N$-island property.
	Since by assumption $\gamma(\lam_0)$ has infinitely many preimages, let $z_0, \ldots, z_N$ denote $N+1$ such distinct preimages.
	We apply Lemma \ref{lem:marked} to the map $F(\lam,z):=f_\la(z)$ and to the $z_i$, thus there exists a 
	branched cover $\pi: \D \to V$,  with $\pi(0)=\lambda_0$, where $V$ is a neighborhood of $\lam_0$, and holomorphic maps $t \mapsto x_i(t)$ on $\D$ such that $f_{\pi(t)}(x_i(t))=\gamma \circ \pi(t)$, and $x_i(0)=z_i$.
	In other words, up to replacing the family {$\{f_\lam\}_{\lam \in V}$} by the family $\{f_{\pi( t)}\}_{ t \in \D}$, 
	we may assume that each preimage $z_i(\la)$ moves holomorphically,  satisfying $f_\lam(z_i(\lam))=\gamma(\lam)$. To keep notations light, we will still denote 
	this reparametrized family by $\{f_\lam\}_{\lam \in V}$. Let $D_i$ ($0 \leq i \le N+1$) be Jordan domains with pairwise disjoint closures each containing $z_i$, and let $\delta>0$ be small enough 
	that for all $0 \leq i \leq N$ and $\la \in \D(\lam_0, \delta)$, we have $z_i(\lam) \in D_i$.

	Decreasing $\delta$ if necessary, the function $G(\la):=\psi_\la^{-1} ( f_\la^n(v_\la))$ is open on $\D(\la_0,\delta)$ by Lemma \ref{lem:qr},
	and  $G(\lam_0)=x\in\partial W$.
	It follows that $G(\D(\la_0,\delta))$ contains a disk $\Delta \subset X$ of $d_X$-radius say $\epsilon>0$ centered at $x$.
	By the island property, there exists $0 \leq i \leq N$ and $U \Subset \Delta \cap W$ such that $f: U \to D_i$ is a conformal isomorphism. Up to relabeling, we will assume without loss of generality that $i=0$.

	
	Since $U$ is contained in  the image of $G$, we let $V_1$ denote a connected component of $G^{-1}(U)$ inside $\D(\la_0,\delta)$. If $D_0$
	(and therefore $U$) is small enough, then $V_1$ is a Jordan domain as well.
	Let us now define $H(\la):= f_\la^{n+1}(v_\la)$. Our goal is to show that $\overline{z_0(V_1)} \subset H(V_1) $ so that Lemma \ref{lem:fp} applied to the maps $\lam \mapsto z_0(\lam)$ and $H$ gives the result. 
	
	In order to see this we write
	\[
	f_\la^{n+1}(v_\la) = \varphi_\la \circ f_{\la_0} \circ \psi_\la^{-1} \circ f_\la^n (v_\la) = \varphi_\la \circ f_{\la_0} \circ G(\la),
	\]
	and therefore
	\[
	H(V_1) = \varphi_\la ( f_{\la_0} (G(V_1)))= \varphi_\la(f_{\la_0}(U))=\varphi_\la(D_0).
	\]
	Now since $\delta$ can be taken arbitrarily small, the values of $\la$ can be arbitrarily close to $\la_0$ and therefore $\varphi_\la$ is arbitrarily close to the identity. It follows that $H(V_1)=\varphi_\la(D_0) \simeq D_0$, while $z_0(V_1) \subset z_0(\overline{\D(\la_0,\delta)}) \subset D_0$.  Moreover, $\partial z_0(\overline{\D(\la_0,\delta)}) $ separates the boundaries of these two sets, so the hypotheses of Lemma \ref{lem:fp} can be applied. This gives the existence of $\la'$ arbitrarily close to $\la_0$ such that $f_{\la'}^{n+1}(v_{\la'})=z_0(\la')$, and since $f_{\la'}(z_0({\la'})) = \gamma(\la')$, we do have 
	$f_{\la'}^{n+2}(v_{\la'})=\gamma(\la')$.

\end{proof}

 With the tools above, we can give some additional characterizations of activity of singular values, which will be usefull to prove approximation theorems in the next section.

\begin{prop}[Active singular values]\label{prop:activity}
	A singular value  $v(\lam)$ is active at $\lam_0$ if and only if  one of the following three cases occurs:
	\begin{enumerate}
		\item There exists $n \geq 0$ such that $f_{\la_0}^n(v(\la_0))\in\partial W_{\la_0} $, non persistently.
		\item  There exists an injective sequence of parameters $\lambda_k\ra\lambda_0$,  such that for some sequence of integers $n_k\ra\infty$, 
		$$
		f_{\la_k}^{n_k}(v(\la_k))\in\partial W_{\la_k} .
		$$
		
		\item There exists a neighborhood $U$ of $\lambda_0$ such that the family $(\lam \mapsto f_\la^n(v_\la))_{n \in \N}$ is well defined and not normal in $U$. This case can only occur if $f$ is exceptional.
	\end{enumerate}
\end{prop}

\begin{proof}[Proof of Proposition~\ref{prop:activity}]
	Taking the formal negation of the definition of passivity, we obtain that $v(\lam)$ is active at $\lam_0$ if and only if both of the following conditions are satisfied for all neighborhood $V$ of $\lam_0$: 
	\begin{enumerate}
		\item[(a)] for all $n \geq 0$, there exists $\lam \in V$ such that $f_\lam^n(v_\lam) \in W_\lam$; and
		\item[(b)] the family of holomorphic maps $\{\lam \mapsto f_\lam^n(v_\la) \}$ is either not well-defined on $V$, or it is well-defined but non-normal.
	\end{enumerate}

	It is clear that  condition (3) implies both (a) and (b), and that conditions (1) and (2) each imply (b).
	Let us prove that (1) also implies (a). Assume that $f_{\la_0}^n(v(\la_0)) \in \partial W_{\la_0}$ non-persistently. Let $G(\la):=\psi_\la^{-1} \circ f_{\la}^n(v(\la))$. 
	Since $W_\la:=\psi_\la(W)$, we have $G(\la) \in W \Leftrightarrow  f_{\la}^n(v(\la)) \in W_\la$, 
	and $G(\la) \in \partial W \Leftrightarrow  f_{\la}^n(v(\la)) \in \partial W_\la$.
	By Lemma \ref{lem:qr}, the map $G$ is either open or constant near $\la_0$; and $G(\la)$ is non-constant since by assumption $f_{\la_0}^n(v(\la_0)) \in \partial W_{\la_0}$ non-persistently.
	Therefore, there exists $\la \in V$ such that $G(\la) \in W$, hence $ f_{\la}^n(v(\la)) \in W_\la$.
	Now that we know that (1) implies (a), it is clear that (2) also implies (a). We have therefore proved 
	that $v(\la)$ is active at $\la_0$ if one of the three cases (1), (2) or (3) occurs.  Let us now prove that case (3) can only occur if $f_{\la_0}$ is exceptional.

	Suppose that (3) holds. 	Then  $X$ cannot be hyperbolic by Lemma \ref{lem:normalf}, and therefore  $X=\rs$ or a complex torus. But endomorphisms of complex tori have no singular values by Hurwitz's formula, so this last possibility is in fact excluded; therefore, $X=\rs$. 
	
	Assume for a contradiction that $f_{\la_0}$ is not exceptional. Then in particular $\bigcup_{n \geq 0} f_{\la_0}^{-n}(\partial W)$ is infinite, by Lemma \ref{lem:boundary}, so there exists $z_1, z_2, z_3 \in \rs$ three distinct points such that $f_{\la_0}^N(x_1)=f_{\la_0}^N(x_2)=f_{\la_0}^N(x_3)=:y\in \partial W_{\la_0}$ for some $N \geq 1$. Let $D \subset M$ be a one-dimensional disk passing through $\lam_0$ such that 
	$\{\la \mapsto f_\la^n(v_\lam): n \in \N\}$ is still well-defined but non-normal on $D$.  By Lemma \ref{lem:marked} applied with $F(\lam,z):=f_\lam^N(z)$ and $\gamma(\lam):=\psi_\lam(y)$, there exists a neighborhood $V$ of $\lam_0$ in $D$, a branched cover
	$\pi: \D \to V$ and holomorphic maps $x_i: \D \to \rs$ such that for all $t \in \D$,
	$$f_{\pi(t)}^N(x_i(t)) = \psi_{\pi(t)}(y).$$
	The family $\{t \mapsto f_{\pi(t)}^n(v_{\pi(t)}) : n \in \N \}$ is non-normal on $\D$, 
	so by Montel's theorem it cannot omit the three moving values $x_1(t), x_2(t), x_3(t)$; therefore, there exists $t_1 \in \D$ and $n \in \N$  such that say $f_{\pi(t_1)}^n(v_{\pi(t_1)})=x_1(t_1)$, 
	which means that $f_{\lam_1}^{N+n}(v_{\lam_1}) \in \partial W_{\lam_1}$, where $\lam_1:=\pi(t_1)$.
	But this contradicts the assumption that $\{\lam \mapsto f_\lam^n(v_\lam) : n \in \N\}$ is well-defined on $D$, hence on $V$. Therefore, $f$ is exceptional. 
	
	Conversely, assume that both (a) and (b) hold. There are two possibilities: 
	first, if there exists a neighborhood $V$ such that $\{\lam \mapsto f_\lam^n(v_\la) \}$ is well-defined but not normal, then we are in case (3).
	Assume from now on that this is not the case, and let $G_n(\lam):=\psi_\lam^{-1} \circ f_\la^n(v_\lam)$ as above.  Then,  for all neighborhood $V$ of $\lam_0$, there exists $n \in \N$ such that 
	$G_n(V) \cap W_{\la_0} \neq \emptyset$  (by (a)) and $G_n(V) \cap (X \setminus W_{\la_0}) \neq \emptyset$ (because $\{\lam \mapsto f_\lam^n(v_\la) \}$ is not well defined), so that $G_n(V) \cap \partial W_{\la_0} \neq \emptyset$.  
	 By considering a basis of neighborhoods $(V_k)_{k \in \N}$ of $\lam_0$, we obtain a sequence $\lam_k \to \lam_0$ (not necessarily injective) and a sequence of integers $(n_k)_{k \in \N}$ (not necessarily unbounded)
	such that $f_{\lam_k}^{n_k}(v_{\lam_k}) \notin W_{\lam_k}$.
	If the sequence $(n_k)$ is bounded, then up to extraction it is constant equal to some $N \in \N$; and by continuity we have $f_{\lam_0}^N(v_{\lam_0}) \in \partial W$. By (a) this relation is not persistent on $M$, and so we are in case (1).

	If the sequence $(n_k)$ is unbounded, then up to extraction we can assume that it is strictly increasing. Then the sequence $(\lam_k)$ must be injective, and so we are in case (2).

\end{proof}

\subsection{Density Theorems}
In this section we prove  that given the activity locus $\AA(v_\la)$ of a singular value $v_\la$, parameters for which $v_\la$ has a Misiurewicz relation and parameters  for which the orbit of $v_\la$ lands on the boundary of $W_\la$ are dense in $\AA(v_\la)$. We also show that $\AA(v_\la)$ is nowhere locally  contained in a proper analytic subset of $M$.

\begin{defi}[Misiurewicz relation]
	Let $\{f_\lam\}_{\lam \in M}$ be a natural family of holomorphic maps, and $\lam_0 \in M$.	
	 We say that $f_{\lam_0}$ has a Misiurewicz relation if there exists 
	a singular value $v_{\lam_0}$, $n \in \N$  and a repelling periodic point $z_{\lam_0}$ such that   $f_{\lam_0}^n(v_{\lam_0})=z_{\lambda_0}$.
\end{defi}


We say that a Misiurewicz relation is persistent if it holds on a parameter neighborhood of $\lambda_0$.

%

\begin{lem}[Misiurewicz relations imply activity]\label{lem:misiuimpliesactive}
	Let $\{f_\lam\}_{\lam \in M}$ be a natural family of Ahlfors island  maps.	Let $\lam_0\in M$ be such that $f_{\lam_0}$ has a Misiurewicz relation, i.e. 	there exists $v_{\lam_0} \in S(f_{\lam_0})$ and $n \in \N$ such that 
	$f_{\lam_0}^n(v_{\lam_0})$ is a repelling periodic point, and this relation is not persistent.
	Then $v_\lam$ is active at $\lam_0$.
\end{lem}

\begin{proof}
	By definition of activity, we may assume without loss of generality that 
	there exists a neighborhood $V$ of $\lam_0$ on which $\{\lam \mapsto f_\lam^{ k}(v_\lam): k \in \N\}$
	are well-defined (otherwise, $v_\la$ is active at $\lam_0$ and we are done).
	Then the proof is the same as in the classical case of e.g. rational maps.
	We reproduce it here for the convenience of the reader.
	
	Let $p$ denote the period of the repelling cycle.
	There exists a neighborhood $U$ of $\lam_0$ such that the repelling periodic point $z_{\lam_0}=f_{\lam_0}^n(v_{\lam_0})$ moves holomorphically over $U$ as $\lam \mapsto z_\lam$, and remains repelling. 
	Moreover, there exists $r>0$ such that the cycle of $z_\lam$ is linearizable on $\D(z_\lam,r)$, that is, there exists local biholomorphisms $\zeta_\lam: \D(0,r) \to W_\lam$ depending holomorphically on $\lam$, such that $\zeta_\lam(0)=z_\lam$ and $f_\lam^p \circ \zeta_\lam(z)=\zeta_\lam(\rho_\lam z)$, where $\rho_\lam$ is the multiplier of the repelling cycle.
	Let $u(\lam):=\zeta_\lam^{-1}(f_\lam^n(v_\lam))$.
	Then $f_\lam^{n+kp}(v_\lam)=f_\lam^{kp} \circ \zeta_\lam(u(\lam))=\zeta_\lam( \rho_\lam^k u(\lam) )$.
	But since by assumption, $u(\lam_0)=0$ and $u$ does not vanish identically and $|\rho_\la|>1$, it is clear that the family $\{\lam \mapsto f_\lam^{n+kp}(v_\lam) \}_{k \in \N}$ is not normal at $\lam_0$, hence that $v_{\lam}$ is indeed active at $\lam_0$. 
 \end{proof}

\begin{lem}[Activity loci are not contained in analytic subsets]\label{lem:perturbA}
	Let $\{f_\lam\}_{\lam \in M}$ be a natural family of Ahlfors island  maps, and let $\mathcal{A}(v_\lam)$
	be the activity locus of a singular value $v_\lam$. Then $\mathcal{A}(v_\lam)$ is nowhere locally contained in a proper analytic subset of $M$. More precisely, if $\lam_0 \in \mathcal{A}(v_\lam) \cap H$, where $H \subset M$ is a proper analytic subset, then for every neighborhood $U$ of $\lam_0$ in $M$, $U \cap (\mathcal{A}(v_\lam) \setminus H) \neq \emptyset$.
\end{lem}

\begin{proof}
	Let $\lam_0 \in \mathcal{A}(v_\lam)$, $H$ a proper  analytic subset   of $M$ containing $\lam_0$, and $U$ be a small polydisk centered at $\lam_0$ in $M$. Assume for a contradiction that 
	$\mathcal{A}(v_\lam) \cap U \subset H \cap U$. Let $h_n(\lam):=f_\lam^n(v_\lam)$, wherever this expression is well-defined. Let $z_{\lam_0}$ be a repelling periodic point of period at least 3 for $f_{\lam_0}$ which is not Picard exceptional. Let $z_\lam$ be the corresponding repelling periodic point for $f_\lam$ given by the Implicit Function Theorem. Up to reducing $U$, we may assume that $\lam \mapsto z_\lam$ is 
	defined over $U$.

	Since $\lam_0 \in \acal(v_\lam)$, there is no $N \in \N$ such that $h_N(\lam) \in X \setminus W_\lam$ for all  $\lam \in U$.
	By Lemma \ref{lem:Gct} and the assumption that $v(\la)$ is active at $\la_0$, 
	we cannot have that $f_\la^n(v_\la) \in \partial W_\la$ for all $\la$ in an open subset of $U$.
	Therefore, if $\lam \in U$ and $n \in \N$ are such that $h_n(\lam) \in \partial W_\lam$, then $\lam \in \acal(v_\la) \cap U$, therefore in $H$.
	
	We now distinguish two cases: 
	\begin{enumerate}
		\item either there exists $n_0 \in \N$ and $\lam_1 \in U$ such that $h_{n_0}(\lam_1) \in \partial W_{\lam_1}$ non-persistently;
		\item or for every $n \in \N$, $h_n$ is well-defined over $U$ but not normal.
	\end{enumerate}

	Let us first treat case (1). Let $D$ be a one-dimensional holomorphic disk  passing through $\lam_1$ and not contained in $H$. Then by the choice of $\lam_1$ and  our previous observation, $h_{n_0}(\lam_1)\in \partial W_{\lam_1}$ and there exists $\lam \in D \setminus \{\lam_1\}$ such that $h_n(\lam) \notin \partial W_\lam$. By the Shooting Lemma (Proposition \ref{lem:shooting}) applied with $\gamma(\lam):=z_\lam$ and $M:=D$, we find some $\lam_2 \in D \setminus \{\lam_1\}$ such that 
	$h_{n_0+1}(\lam_2)=z_{\lam_2}$, in other words, $v_{\lam_2}$ is Misiurewicz.  Therefore $\lam_2 \in \acal(v_{\lam_2})$, but $\lam_2 \notin H$, a contradiction.

	Case 2 follows from a similar but more classical application of Montel's theorem.
\end{proof}

\begin{prop}[Density of truncated parameters]\label{prop:trunc} Assume that $f_{\lambda_0}$ is a non-exceptional Ahlfors island map.
	Let $v_\lam$ be a singular value, and assume that it is active at $\lam_0$.
	Then there exists $n_k \to +\infty$ and $\lam_k \to \lam_0$ such that $f^{n_k}(v_{\lam_k})\in\partial W_{\lambda_k}$ non persistently.
\end{prop}

\begin{proof}
	Given $N \in \N$, we will construct $\lam$ arbitrarily close to $\lambda_0$ such that $f_{\lam}^n(v_{\lam}) \in \partial W_\lam$, for some $n \geq N$.
	In view of Lemma \ref{lem:densepreimages} and since $f$ is not exceptional, the set $\bigcup_{n \geq 0}  f^{-n}(\partial W)$ is infinite.
By the Ahlfors Island property we may find $n \geq N$ and  $x \in f^{-n}(\partial W)$ with infinitely many preimages.
	By Lemma \ref{lem:marked} applied to $F(\lam, z):=f_\lam^n(z)$
	and $\gamma(\lam):=\psi_\lam \circ f^n(x) \in \partial W_\lam$,
	up to passing to a branched cover in parameter space, 
	we may assume without loss of generality that there is a local holomorphic germ $\lam \mapsto x_\la$, with $x_{\lam_0}=x$ 
	and for all $\lam$ close enough to $\lam_0$, 
	$f_\lam^n(x_\lam) \in \partial W_{\la}$.

	By Proposition ~\ref{prop:activity} and since $v_\lam$ is active at $\lam_0$, we may assume without loss of generality  that there is $n_0 \in \N$ such that 
	$f_{\lam_0}^{n_0}(v_{\lam_0}) \in \partial W_{\lam_1}$ non-persistently. Applying Lemma \ref{lem:shooting}
	with $\gamma(\lam):=x_\la$, we find $\lam_1$ arbitrarily close to $\lam_0$ such that 
	$f_{\lam_1}^{n_0+2}(v_{\lam_1}) = x_{\lam_1}$, which implies $f_{\lam_1}^{n_0+2+n}(v_{\lam_1}) = y_{\lam_1} \in \partial W_{\lam_1}$. 
	
\end{proof}

\begin{prop}[Density of Misiurewicz parameters]\label{prop:misiu}
	Let $v_\lam$ be a singular value, and assume that it is active at $\lam_0$.
	Let $x_{\lam_0}$ be a repelling periodic point for an Ahlfors island map $f_{\lam_0}$ of period at least 3 which is not a Picard exceptional value, and let $x_\la$ be its analytic continuation in some neighborhood of $\lambda_0$. Then there is $n_k \to +\infty$ and   $\lam_k \to \lam_0$ such that 
	$$f_{\lam_k}^{n_k}(v_{\lam_k}) = x_{\lam_k}.$$ 
\end{prop}

\begin{proof}
	Let us first assume that $f$ is not exceptional. Let $\eps>0$ and $N \in \N$, and let $\lam \mapsto x_\lam$ 
	denote the holomorphic motion of $x_{\lam_0}$ as a repelling periodic point, which we may assume to be well-defined for  $\lam \in \B(\lam_0,\eps)$ (up to taking $\eps>0$ small enough).
	By Proposition \ref{prop:trunc}, there exists $\lam_1 \in \B(\lam_0, \frac{\eps}{2})$ and $n_1>N$ such that 
	$f_{\lam_1}^N(v_{\lam_1}) \in \partial W_{\lam_1}$ (non-persistently). By Proposition \ref{lem:shooting},
	there exists $\lam_2 \in \B(\lam_1,\frac{\eps}{2})$ such that $f_{\lam_2}^{n_1+2}(v_{\lam_2}) =x_{\lam_2}$, 
	and we are done.
	
	Asssume now that $f$ is exceptional. By Proposition \ref{prop:sparsexcep}, either for all $\lam \in M$ we have that 
	$f_\lam$ is exceptional, or the set of $\lam \in M$ such that $f_\lam$ is exceptional is a proper analytic subset of $M$. In the latter case, we may use Lemma \ref{lem:perturbA} to perturb slightly $\lam_0$ to remain in the activity locus of $v_{\lam_0}$ but outside this analytic set, thus reducing to the non-exceptional case.
	Finally, if all maps $f_\la$ are exceptional, then we can just apply the classical argument using Montel's theorem.
\end{proof}

\section{Characterization of stability: Proof of Theorem \ref{th:mssahlfors}}

In this section we prove that the backward orbits of repelling periodic points move holomorphically, provided they do not collide with the postsingular set. We then use this fact to prove Theorem~\ref{th:mssahlfors}.
In what follows $M$ always denotes a connected complex manifold.

\subsection{Generalities on $J$-stability: Proof of Proposition \ref{prop:equivalent_definitions}}

\begin{defn}[Holomorphic motions respecting the dynamics]\label{def:holomotion}
A {\em holomorphic motion} of a set $A \subset X$  over an open set $U\subset M$
 with basepoint $\la_0\in U $ is a map  $H: U  \times A \rightarrow X$ given by 
 $(\la,x) \mapsto H_\la(x) $ such that
 \begin{enumerate}
\item  for each $x \in A$ , $\la \mapsto H_\la(x)$ is holomorphic , 
 \item   for each  $\la \in U$,  $H_\la(\cdot) $ is injective on $A$, and,
  \item  $H_ {\la_0} \equiv {\rm Id}$.
  \end{enumerate}

Recall that, by the $\lambda$-lemma \cite{mss}, a holomorphic motion is jointly continuous, and $H_\la$ extends to a continuous map of $\overline{A}$.

A holomorphic motion of a set $X$ {\em respects the dynamics} of the holomorphic  family $\{f_\la\}_{\la\in M}$ if 
 $$
 H_\la \circ f_{\la_0} = f_\la \circ H_{\la}
 $$ whenever both $x$ and $f_{\la_0}(x)$ belong to $A$.
\end{defn}

Recall our definition of $J$-stability.

\begin{defi}\label{def:jstablesection}
	Let $\{f_\la = \phi_\la \circ f_{\la_0} \circ \psi_\la^{-1}\}_{\la \in U}$ be natural family of Ahlfors island maps with basepoint $\la_0 \in U$. We say that $\{f_\la\}_{\la \in U}$ is $J$-stable if there exists a holomorphic motion 
	$H: U \times J(f) \to X$ with basepoint $\la_0$ such that 
	\begin{enumerate}
		\item\label{item1} for all $\la \in U$, $h_\la:= H(\la, \cdot): J(f_{\la_0}) \to J(f_\la)$ and 
		$h_\la \circ f_{\la_0} = f_\la \circ h_\la$
		\item\label{item2} for all $\la \in U$, $h_\la = \phi_\la$ on $S(f_{\la_0}) \cap J(f_{\la_0})$.
	\end{enumerate}
\end{defi} 

Proposition \ref{prop:equivalent_definitions} in the Introduction states that, in the case of finite maps or when $f$ is meromorphic and the set of singular values has no interior, this definition coincides with the classical one. This will be the content of Lemma \ref{lem:compatibilityftm} and Proposition \ref{prop:S(f)nointerior} respectively, which we prove in this section.

\begin{lem}\label{lem:compatibilityftm}
	Let $\{f_\la = \phi_\la \circ f_{\la_0} \circ \psi_\la^{-1}\}_{\la \in U}$ be a natural family of finite type maps.
	Assume that there is a a holomorphic motion 
	$H: U \times J(f) \to X$ with basepoint $\la_0 \in U$ satisfying (\ref{item1}). Then $H$ satisfies (\ref{item2})
	and so the family $\{f_\la\}_{\la \in U}$ is $J$-stable in the sense of Definition \ref{def:jstablesection}.
\end{lem}

\begin{proof}
	Let $f:=f_{\la_0}$, and $v \in S(f) \cap J(f)$. Let $h_\la:=H(\la,\cdot): J(f) \to J(f_\la)$.
	It is enough to prove that for all $\la \in U$, $h_\la(v) \in S(f_\la)$. Indeed, if it is the case, then since $h_{\la_0}(v) = \phi_{\la_0}(v)=v$ and $\phi_\la(v) \in S(f_\la)$ for all $\la \in U$, 	we must have by $h_\la(v)=\phi_\la(v)$ by continuity and finiteness of $S(f_\la)$.
	
	Let us now prove that $h_\la(v) \in S(f_\la)$. We have two (non-mutually exclusive) cases to treat: either $v$ is a critical value or it is an asymptotic value with a logarithmic tract. In the first case, there is a critical point $c \in J(f)$ such that $f(c)=v$. Then $h_\la(c) \in J(f_\la)$ is a critical point for $f_\la$ and $f_\la(c)=h_\la(v)$ is a critical value for $f_\la$.
	
	Let us now treat the case when  $v$ is an asymptotic value: then there exists a topological disk $D \subset X$ containing $v$ and a connected component $V$ of 
	$f^{-1}(D)$ such that $f: V \to D \setminus \{v\}$ is a universal cover. Let $\la \in U$, and let $D_\la$ be a topological disk containing $h_\la(v)$
	chosen small enough that $D_\la^* \cap S(f_\la) = \emptyset$ where $D_\la^*:=D_\la \setminus \{h_\la(v)\}$, and that $h_\la^{-1}(D_\la \cap J(f_\la)) \subset D$.

	Let  $z \in V\cap J(f)$, and let $V_\la$ denote the connected component of $f_\la^{-1}(D_\la)$ which contains $h_\la(z)$. We claim that $V_\la$ contains no preimage of $h_\la(v)$: indeed, if $y$ were such a point, then $h_\la^{-1}(y)$ would be a preimage of $v$ in $V$ for $f$, which is impossible.  By the classification of coverings of the punctured disk $D_\lambda^*$,  it follows that $f_\la: V_\la \to D_\la^*$ is equivalent to the exponential map; therefore $h_\la(v)$ is an asymptotic value, and we are done.

\end{proof}

The following proposition gives the same conclusion but for natural families of meromorphic maps whose singular value set has no interior.

\begin{prop}\label{prop:S(f)nointerior}
	Let $\{f_\la = \phi_\la \circ f_{\la_0} \circ \psi_\la^{-1}\}_{\la \in U}$ be a natural family of \emph{meromorphic}  maps, 
	such that $S(f_\la)$ has no interior.
	Assume that there is a a holomorphic motion 
	$H: U \times J(f) \to X$ with basepoint $\la_0 \in U$ satisfying (\ref{item1}). Then $H$ satisfies (\ref{item2})
	and so the family $\{f_\la\}_{\la \in U}$ is $J$-stable in the sense of Definition \ref{def:jstablesection}.
\end{prop}

The proof is based on the following result from Heins.

\begin{theo}[\cite{heins}, Theorem 4]\label{th:heins}
	Let $f: \C \to \rs$ be a non-constant meromorphic map, and let $V \subset \rs$ be a domain. Let $U$ be a connected component of $f^{-1}(V)$. Then either $f: U \to V$ has constant finite degree, or there are at most 2 values in $V$ with finitely many preimages in $U$.
\end{theo}

Before proving Proposition \ref{prop:S(f)nointerior}, we will require the following lemmas.

\begin{lem}\label{lem:key}
	Let $f$ be a transcendental meromorphic map, let $V \subset \rs$ be a domain such that $V \cap J(f) \neq \emptyset$, and let $U$ be a connected component of 
	$f^{-1}(V)$. Then $f: U \to V$ is a biholomorphism if and only if $f: U \cap J(f) \to D \cap J(f)$ is a homeomorphism. 
\end{lem}

\begin{proof}
	Clearly, if $f: U \to V$ is a biholomorphism then  $f: U \cap J(f) \to D \cap J(f)$ is a homeomorphism. Conversely, let us assume that  $f: U \cap J(f) \to V \cap J(f)$ is a homeomorphism. Since $J(f)$ is perfect, $V \cap J(f)$ is infinite;  then every $z \in J(f) \cap V$ has exactly one preimage in $U$ by assumption and by invariance of the Julia set, so $f: U \to V$ must have constant finite degree by Theorem \ref{th:heins}. That degree is then 1 and hence $f: U \to V$ is biholomorphic.
\end{proof}

\begin{lem} \label{lem:Sinvariant}
	Let $\{f_\la\}_{\la \in M}$ be a natural family of meromorphic maps. 
		Assume that there is a a holomorphic motion 
	$H: U \times J(f) \to X$ with basepoint $\la_0 \in U$ satisfying (\ref{item1}). 
	Then for all $\la \in M$, $h_\la$ maps $S(f) \cap J(f)$ to $S(f_\la) \cap J(f_\la)$.
\end{lem}

\begin{proof}
	Let $v \in J(f) \cap S(f)$. By definition, $h_\la(v) \in J(f_\la)$, so we must prove that $h_\la(v) \in S(f_\la)$.
	Let $D$ be a topological disk containing $v$ and $U$ be a connected component of $f^{-1}(D)$ such that $f: U \to D$ is not a biholomorphism.
	
	Let $D_\la$ be a topological disk containing $h_\la(v)$, chosen small enough that $h_\la^{-1}(D_\la) \subset D$.
	Let $z \in U \cap J(f)$ chosen such that $h_\la \circ f(z) \in D_\la$ (this is possible since $J(f)$ is perfect and since by Theorem \ref{th:heins}, there are only finitely many $y \in D$ with no preimage in $U$). Finally, let $U_\la$ denote the connected component of $f_\la^{-1}(D_\la)$ containing $h_\la \circ f(z)$.
	By Lemma \ref{lem:key}, $f: U \cap J(f) \to D \cap J(f)$ is not a homeomorphism. 
	Since $h_\la: J(f) \to J(f_\la)$ is a topological conjugacy, $f_\la: h_\la(U) \cap J(f_\la) \to D_\la \cap J(f_\la)$ 
	is not a homeomorphism either.
	Again by the lemma above, $f_\la:U_\la \to D_\la$ is not a biholomorphism.
	
	We have proved that for any small enough disk $D_\la$ around $h_\la(v)$, there exists a connected component $U_\la$ 
	such that $f_\la: U_\la \to D_\la$ is not a biholomorphism; therefore, $h_\la(v) \in S(f_\la)$.
\end{proof}

\begin{proof}[Proof of Proposition \ref{prop:S(f)nointerior}]
Let $v\in S(f)\cap J(f)$ and define $g(\la):=\varphi_\la^{-1} \circ h_\la(v)$ which maps into $S(f)$ by Lemma \ref{lem:Sinvariant}. Let $U'\subset U$ be an arbitrary one-dimensional section of $U$ containing $\la_0$. Then, by Lemma \ref{lem:qr}, $g|_{U'}$ is open or constant. But  $S(f)$ has empty interior, hence $g(\la)\equiv v$ for all $\la\in U'$, since $g(\la_0)=v$.
Since $U'$ was arbitrary, the claim follows.
\end{proof}

\subsection{Independence of the notion of $J$-stability from the choice of $\phi_\lambda$: Proof of Proposition \ref{prop:indepe}}\label{sec:independence on phi}
\

 The following is a restatement of Proposition \ref{prop:indepe} in the introduction.

\begin{prop}[The notion of $J$-stability does not depend on  $\phi_\lambda$]\label{prop:independence on phi}
	Let $\{f_\la\}_{\la\in M}$ be a natural family of Ahlfors island maps, and assume that 
	$$f_\la = \phi_\la \circ f_{\la_0} \circ \psi_\la^{-1} = \tilde \phi_\la \circ f_{\la_0} \circ \tilde \psi_\la^{-1}$$
	where $f_{\la_0}: W \to X$ and $\partial W$ is totally disconnnected.
	Then $\tilde \phi_\la(v) = \phi_\la(v)$ for  all $\la\in M $ and for all $v \in S(f)$ .
\end{prop}

 Observe that $\partial W$ is totally disconnected for all transcendental meromorphic maps ($\partial W=\{\infty\}$) or iterates of such, which belong to the class $\mathbb{K}$ of maps that  are holomorphic outside a countable set of points (see \cite{bolsch, bolsch2}). See also  \cite{her98,BDH,BDH2}.

Proposition~\ref{prop:independence on phi} is a consequence of the following Lemma.

\begin{lem}\label{lem:accesses}
	Let $f: W \to X$ be an Ahlfors isand map with $\partial W$ totally disconnected.
	Let $D \subset X$ be a topological disk, and $U$ a connected component of $f^{-1}(D)$. 
	Assume that there is an asymptotic path  $\gamma:  [0,1) \ra  U$ such that $f\circ \gamma(s) \to v \in D$ as $s\ra 1$, 
	and an isotopy $h_t: X \to X$ of homeomorphisms depending continuously on $t \in [0,1]$ such that 
	\begin{enumerate}
		\item $h_0 = \id$
		\item $h_t \circ \gamma(s) \in U$ for all $t,s \in [0,1]$
		\item $h_t$ is the identity on $\partial W$ for all $t$
		\item $f\circ h_t \circ \gamma(s) \to v_t$ as $s \to  1$, for all $t$.
	\end{enumerate}
	Then $v_1=v$.
\end{lem}

We delay its proof for the end of the section.

\begin{proof}[Proof of Proposition~\ref{prop:independence on phi}, assuming Lemma \ref{lem:accesses}]
	Let $f:=f_{\la_0}$. We may assume without loss of generality that $\phi_{\la_0}=\psi_{\la_0}=\tilde \phi_{\la_0} = \tilde \psi_{\la_0}=\id$ (see Remark \ref{rem:marking}).
Let $\la_1 \in M$ be any parameter, and let $t \mapsto \la_t$ be a continuous path in $M$ joining $\la_0$ to $\la_1$. Let  $\phi_t:=\phi_{\la_t}^{-1} \circ \tilde \phi_{\la_t}$ and $\psi_t:=\psi_t^{-1} \circ \tilde \psi_{\la_t}$. Then $\phi_t$ and $\psi_t$ depend continuously on $t \in [0,1]$, and 
we have $f = \phi_t \circ f \circ \psi_t^{-1}$. We must prove that for all $v \in S(f)$ and all $t \in [0,1]$, $\phi_t(v)=v$.

 For each $t \in [0,1]$, $\phi_t$ maps $S(f)$ to $S(f)$, and $\psi_t$ maps critical points of $f$ to critical points of $f$. Since the set of critical points is discrete, we must have $\psi_t=\id$ on the set of critical points of $f$ for all $t$, and therefore $\phi_t(v)=v$ for all $t \in [0,1]$ if $v$ is a critical value.
It remains to treat the case where $v$ is an asymptotic value.

Let $v$ be an asymptotic value for $f$, $D$ be a small disk near $v$ and $\gamma$ be an asymptotic path for $v$ under $f$ contained in some $U$ unbounded connected component of $f^{-1}(D)$.
By definition, any accumulation point of $\gamma$ belongs to $\partial W$; and since by assumption $\partial W$ is totally disconnected, $\gamma$ must in fact converge in $X$ to a limit point $x_0 \in \partial W$.
We claim that $\psi_t=\id$ on $\partial W$. Indeed, for all $\la \in M$, $\psi_\la^{-1} \circ \tilde \psi_\la$ maps $\partial W$ to $\partial W$; 
and by Lemma \ref{lem:qr}, for any $z \in \partial W$, the map $\la \mapsto \psi_\la \circ \tilde \psi_\la^{-1}$ is either open or constant.
But it cannot be open since $\partial W$ has no interior. Therefore we have  $\psi_\la^{-1} \circ \tilde \psi_\la(z)=z$ for all $z \in \partial W$ and all $\la \in M$,
and in particular $\psi_t=\id$ on $\partial W$ for all $t \in [0,1]$.  In particular this implies that $\psi_t(\gamma(s)) \to x_0$ when $s\to 1$.

Additionally, since $f \circ \gamma(s) \to v$ and $f=\phi_t^{-1} \circ f \circ \psi_t$, we also have $f \circ \psi_t \circ \gamma(s)  \to \phi_t(v)$  as $s\to 1$, for all $t \in [0,1]$.

Now we want to apply  \ref{lem:accesses} with $h_t:=\psi_t$ and $v_t:=\phi_t(v)$. To do so, we need to show that the isotopic curves $\psi_t(\gamma)$ belong to $U$, which is not obvious, since $\psi_t$ is the identity on $\partial W$ but not on $\partial U$. To fix this, choose $D_0\Subset D$ small enough so that $\cup_{t\in[0,1]} \overline{D_t} \Subset D_0$, where $D_t:=\varphi_t(D_0)$.   Let $U_t$ be the connected component of $f^{-1}(D_t)$ contained in $U_0$ and observe that necessarily 
\[
\bigcup_{t\in[0,1]} \left(\overline{U_t}\setminus \{x_0\}\right) \Subset U.
\]
We now restrict the curves, that is, choose  $\epsilon >0$ small enough such that $f\circ \psi_t \circ \gamma[\epsilon,0) \subset D_t$, which is possible since the curves converge to $v_t$. It follows that 
\[
\psi_t(\gamma[\epsilon,0)\subset U_t\subset U,
\]
and hence the curves $\psi_t(\gamma(s))$ are isotopic in $U$, with fixed endpoint at $x_0$.

We can now apply Lemma \ref{lem:accesses} applied with $h_t:=\psi_t$ and $v_t:=\phi_t(v)$, obtaining that  $\phi_t(v)=v$ for all $t \in [0,1]$, as we wanted to show.

\end{proof}

To prove Lemma~\ref{lem:accesses} we use two results. The first one is by Lehto and Virtanen and can be found in  \cite{LehtoVirtanen}, or in \cite[Sect.~ 4.1]{Pom92}.

\begin{theo}[Lehto-Virtanen]
Let $g$ be a holomorphic  map  from $\D$ to  $\chat$ omitting three points. Let $\gamma$ be a curve in $\D$ landing at $\zeta\in\partial \D$. If $g(\gamma)$ lands at $p\in\C$, then $g$ has angular limit (and in particular radial limit) at $\zeta$ equal to $p$. It follows that, if two curves in $\D$ land at the same point in $\partial \D$, their images under $g$ cannot land at different points of $\C$.
\end{theo}

The second is part of a correspondence theorem between accesses to points in the boundary of a (possibly multiply connected) domain, and points in the boundary of $\D$ under the universal covering map. It can be found \cite[Theorem 5.9]{FerreiraJove} and is a generalization of the correspondence theorem for simply connected domains in \cite{bfjk_accesses}. We first give the definition of {\em access} (c.f. \cite[Definition 5.7]{FerreiraJove}).

\begin{defn}[Access] Let  $U\subset X$ be a domain, $z_0\in U$ and  $p\in\partial U$. A homotopy class  of curves  $\eta: [0, 1] \ra X$,   fixing the endpoint $\eta(1)$,  such that  $\eta([0,1))\subset U$ and  $\eta(1)=p$  is called an access to $p$ in $U$. 
\end{defn}

\begin{theo}[Correspondence Theorem for multiply connected domains]\label{thm:correspondence FerreiraJove}
Let $U \subset X$ be a hyperbolic  domain and $\pi\colon\D\to U$ a universal covering map. Suppose $p\in \partial U$ and let $A$ be an access to $p$ in $U$.  Then  there exists $\zeta\in \partial \D$, such that  any curve $\eta\in A$, has a  lift $\tilde\eta\subset\pi^{-1}(\eta)$  landing at $\zeta$.
\end{theo}

 We are now ready to proof Lemma~\ref{lem:accesses}.

\begin{proof}[Proof of  Lemma~\ref{lem:accesses}]
Let $\pi:\D \to U$ be a universal covering map  and let $\gamma$ and  $h_t$ be as in the statement. Asymptotic paths  can only accumulate at $\partial W$ which is totally disconnected, hence there exists $p\in\partial W$ such that $\gamma(s)\to p$ as $s\to 1$.  Likewise, $\gamma_t(s)=h_t(\gamma(s))\to p$ as $s\to 1$ since $h_t$ is the identity on $\partial W$. Hence the curves are isotopic in $U$ through an isotopy that fixes the endpoint $p$, and therefore they belong to the same access to $p$ in $U$. 

 By Theorem~\ref{thm:correspondence FerreiraJove},  for every $t\in [0,1]$ there is lift  $\tilde \gamma_t \subset \D$ under $\pi$ which lands at some point $\zeta\in \partial \D$, independent of $t$.  Consider $g:=f\circ \pi:\D \to D$. By assumption  we have that $g\circ\tilde\gamma_t=f(\gamma_t)$ lands for every $t$,   hence they must do so at the same point by the Lehto-Virtanen Theorem.
\end{proof}

\subsection{Holomorphic motion of backward orbits}

Let 
$$
\mathcal{O}^-({w,g}):=\bigcup_{n\geq0}g^{-n}(\{w\})
$$
 denote the backwards orbit of $w$ under the map $g$.

We define the postsingular set of $f_\lam$ as $\bigcup_{n \geq 0} f_\lam^n(S(f_\lam))$
(without taking the closure).

\begin{prop}[Holomorphic motion of backward orbits] \label{prop:motiongo}
	Let $\{f_\la:W_\la \to X\}_{\la\in M}$ be a natural family of Ahlfors island maps. Let $z_0$ be a repelling periodic point of period $p\geq 1$  for $f:=f_{\la_0}$. Let $U$ be a simply connected neighborhood of $\la_0$ over which the analytic continuation of $z_0$, denoted by $z_0(\la)$, remains repelling, and suppose that  for all $\lam \in U$, $f_\lam$ has no non-persistent Misiurewicz relations of the form $f_\lam^n(v(\lam)) =z_0(\lam)$, where $v(\la)$ is either a critical or asymptotic value.  Then, there is a holomorphic motion 
	\[
	\begin{array}{rccc}
		H:&U \times \mathcal{O}^-(z_0,f^p) &\longrightarrow &\mathcal{O}^-(z_0(\la),f_\la^p)\\
		& (\la, z)&\longmapsto & z(\la)
	\end{array}
	\]
	preserving the dynamics of $f_\lam^p$.
\end{prop}

\begin{proof} 
	 Let $\varphi_\la, \psi_\la:X\to X$ be the quasiconformal homeomorphisms such that $f_\la=\varphi_\la\circ f \circ \psi_\la^{-1}$ (see Definition \ref{def:natural}). We will prove the statement in several steps.
	We shall first show that for every  choice of $n\geq 1$, the set  $f_\lam^{-n}(\{z_0\})$ moves holomorphically with $\la \in U$. Then, we will prove that there are no collisions between the motion
	of points belonging to $f_\lam^{-m}(\{z_0\})$ and $f_\lam^{-k}(\{z_0\})$   when $k\neq m$ provided both are multiples of $p$. 
	For $n=1$ consider the set 
	\[
	Z_1=\{(\la,z) \mid f_\la(z)=z_0(\la)\},
	\]
	which is an analytic hypersurface of $U\times X$. Let 
	\[
	\pi_1: Z_1\to M
	\]
	denote the projection onto the first coordinate. \\

	\textbf{Step 1.} We first claim that  every irreducible component of $Z_1$ is the graph of a holomorphic map from $U$ to $X$. \\
	
 We will first treat the case where $z_0 \in S(f)$. Let $(\la_1,z_1) \in Z_1$. By assumption, Misiurewicz relations (if there are any) are persistent, and $f_{\la_0}^p(z_0)=z_0$ is such a relation. Therefore for all $\la \in U$, $z_0(\la) \in S(f_\la)$ and,  since $f_\la$ is a natural family,  $z_0(\la) = \varphi_\la(z_0 )$.
	Let $z_1(\la):=\psi_\la \circ \psi_{\la_1}^{-1}( z_1)$. Then, for all $\la \in U$: 
	\begin{align*}
		f_\la(z_1(\la)) &= \varphi_\la \circ f  \circ \psi_\la^{-1} \circ \psi_\la \circ  \psi_{\la_1}^{-1}(z_1)\\
		&= \varphi_\la \circ \varphi_{\la_1}^{-1} \circ \varphi_{\la_1} \circ f \circ \psi_{\la_1}^{-1}(z_1 ) \\
		&= \varphi_{\la}  \circ \varphi_{\la_1}^{-1}  \circ f_{\la_1}(z_1) \\
		&= \varphi_{\la}  \circ \varphi_{\la_1}^{-1} (z_0(\la_1) ) \\
		&= \varphi_\la( z_0 ) = z_0(\la). 
	\end{align*}
	Therefore, for all $\la \in U$, $(\la, z_1(\la)) \in Z_1$, which proves Step 1 in the case where $z_0 \in S(f)$.
	
	We now assume (in the rest of the proof of Step 1) that $z_0 \notin S(f)$. Let $(\lam_1,z_1) \in Z_1$ and let $Z$ denote the irreducible component of $Z_1$ containing $(\lam_1,z_1)$.
	Again by the persistence assumption, $z_0(\la_1)$ is not a critical value, and hence $z_1$ is not a critical point of $f_{\la_1}$.
	By the Implicit Function Theorem, $Z$ is then a complex manifold and $\pi_1: Z \to U$ is locally invertible. It is then well known that $\pi_1$ is a covering unless it has as asymptotic value.

	So suppose that $\la^*$  is an asymptotic value of $\pi_1$, i.e. there exists a path $(\la_t,z_t) \in Z_1$, $t\in [0,1)$ such that $\la_t\to \la^*$ while $z_t\to \partial W_{\la^*}$, as $t\to 1$.  Let $\varphi_t,\psi_t,f_t$ denote respectively $\varphi_{\la(t)}, \psi_{\la(t)}, f_{\la(t)}$. Then, by definition of $Z_1$,
	\[
	f_t(z_t)=(\varphi_t \circ f \circ \psi_t^{-1} ) (z_t)= z_0(\la(t)),
	\]
	and hence
	\begin{equation}\label{defeq}
		f(\psi_t^{-1} (z_t))=\varphi_t^{-1}(z_0(\la(t))).
	\end{equation}
	Now, when $t \to 1$ we have that $z_t\to \partial W_{\la^*}$, hence $\psi_t^{-1} (z_t)\to \partial W_{\la_0}$, since $\psi_{\la^*}(W_0)=W_{\la^*}$. On the other hand, since $z_0(\la(t))\to z_0(\la^*)$ it follows that $\varphi_t^{-1}(z_0(\la(t))) \to \varphi_{\la^*}^{-1}(z_0(\la^*))$, which makes this point an asymptotic value of $f$ by (\ref{defeq}), because $t \mapsto \psi_t^{-1} (z_t))$ is a curve converging to $\partial W_{0}$. Hence $z_0(\la^*)$ is an asymptotic value for $f_{\la^*}$, contradiction.
	Thus $\pi_1: Z \to U$ is a covering map, and since $U$ is simply connected, it is invertible, which implies that $Z$ is a holomorphic graph above $U$. This concludes the proof of Step 1. \\

	\textbf{Step 2.} The irreducible components of $Z_1$ are pairwise disjoint. \\
	 
	Indeed, let us assume for a contradiction that  $Z$ and $Z'$ are two distinct irreducible components of $Z_1$ and $(\la_1, z_1) \in Z \cap Z'$. Then $z_1$ must be a critical point for $f_{\la_1}$, and $z_0(\la_1)$ must be a critical value. But then by the proof of Step 1, both $Z$ and $Z'$ are the graph of the same map $\la \mapsto \psi_\la \circ \psi_{\la_1}^{-1}(z_1)$, so $Z=Z'$, a contradiction.\\

	\textbf{Step 3: Conclusion.} The set $\ocal^-(z_0(\la), f_\la^{p})=\bigcup_{n \geq 0} f_\la^{-np}(\{z_0(\la)\})$ moves holomophically over $U$.\\

	Steps 1 and 2 prove that the set $Z_1$ is a disjoint union of holomorphic graphs $\Gamma_i= \{(\la,z_i(\la)), \la \in U\}$ over $U$, i.e. that the set   $f_\la^{-1}(\{z_0(\la)\})$ moves holomorphically over $U$, with $H_\la(z_i(\la_0)):=z_i(\la)$.

	Now assume we have proven the existence of a holomorphic motion of   $f_\la^{-n+1}(\{z_0(\la)\})$. By considering $ z_n\in f^{-n}(\{z_0\})$ and applying the same arguments, we obtain  by induction that for every $n \in \N$,
	\[
	Z_n:=\{(\la,z)\mid f_\la^n(z)=z_0(\la), \text{but $f_\la^j(z) \neq z_0(\la)$ for any $j<n$}\}
	\]
	is a disjoint union of graphs over $U$, hence also moves holomorphically over $U$. Observe that by construction, 
	this holomorphic motion preserves the dynamics.
	
	To end the proof we need to show that  if $n_1\neq n_2$ then $Z_{n_1 p}$ and $Z_{n_2 p}$ are disjoint sets.  Let $g_\la:=f_\la^p$ so that $z_0(\la)$ is a fixed point of $g_\la$.
	
	So suppose $z(\la)$ and $\tilde{z}(\la)$ satisfy the defining equation
	\[
	g_\la^{n_1}(z(\la)) = g_\la^{n_2}(\tilde{z}(\la)) = z_0(\la),
	\]
	for some $n_1, n_2\in \N$  and assume they coincide at  some $\la=\la^*$, i.e. $z(\la^*)=\tilde{z}(\la^*)=z^*$. 
	
	Let $n=\max(n_1,n_2)$. Since $z_0(\la)$ is fixed for $g_\la$, we have that 
	\[
	g_\la^{n}(z(\la)) = g_\la^{n}(\tilde{z}(\la)) = z_0(\la),
	\]
	for all $\la\in U$. Then, either $z(\la)\equiv \tilde{z}(\la)$ in $U$ and we are done, or $z^*$ is a critical point of $g_{\la^*}^n$,  in which case we would have a non-persistent Misiurewicz relation, impossible by assumption.

\end{proof}

\begin{coro}[$J-$stability] \label{coro:jstable}
	Let $\{ f_\lam\}_{\la \in M}$ be a natural family of Ahlfors island map.
	Let $U \subset M$ be a simply connected domain over which a repelling periodic point $z(\la)$ moves holomorphically, and assume that there are no non-persistent Misiurewicz relations on $U$. 	
	Then the family $\{ f_\lam\}_{\la \in M}$ is $J$-stable on a neighborhood of $\lambda_0$.
\end{coro}

\begin{proof}

	By Lemma \ref{lem:densepreimages}, the set $\mathcal{O}^-(z_0, f^p)$ is dense in $J(f)$.
	Therefore, by the classical $\lam$-lemma, this holomorphic motion extends to a holomorphic motion of $J(f)$ 
	\[
	\begin{array}{rccc}
		H:&U \times J(f) &\longrightarrow &J(f_\la)\\
		& (\la, z)&\longmapsto & H_\la(z):=z(\la)
	\end{array}
	\] which by continuity still preserves the dynamics of $f^p$, i.e.  
	\[
	H_\lam \circ f^p(z)=f_\lam^p \circ H_\lam(z)
	\]
	for all $z \in J(f)$.
	We claim that we must then have in fact 
	$$H_\lam \circ f(z)=f_\lam \circ H_\lam(z)$$
		for all $z \in J(f)$.
	By Theorem \ref{th:repdense} and continuity, it is enough to prove this only for repelling periodic points.
	Let $x$ be a repelling periodic point of period $m$ for $f=f_{\lam_0}$. 
	Reducing $U$ if necessary, let $x_\lam$, $\la\in U$,  denote its local analytic continuation as a  repelling periodic point of period $m$ for $f_\lam$ (given by the Implicit Function Theorem). 
	Since $f^{mp}(x)=x$,
	we have $f_\lam^{mp} \circ H_\lam(x)= H_\lam(x)$ for all $\lam \in U$
	and since $H$ is a holomorphic motion with basepoint $\lam_0$, we have $H_{\lam_0}(x)=x$.
	By continuity of $H$ and since repelling fixed points of $f_\lam^{mp}$ are isolated and move holomorphically, we must have $H_\lam(x)=x_\la$ locally for $\lam$ close to $\lam_0$, 
	and therefore 
		$$H_\lam \circ f(z)=f_\lam \circ H_\lam(z)$$
	for all {repelling} periodic points of $f$ (hence for all $z \in J(f)$ by continuity) and for all $\lam$ in a neighborhood of $\lam_0$.
	The result finally follows from the Identity Theorem applied on $M$ to the holomorphic maps 
	$\lam \mapsto H_\lam \circ f(z)$ and $\lam \mapsto f_\lam \circ H_\lam(z)$, $z \in J(f)$.
\end{proof}

\begin{prop}\label{prop:passivity}
	Let $\{f_\la\}_{\la \in M}$ denote a natural family of Ahlfors island maps. Assume that this family is $J$-stable; then all singular values are passive on $U$.
\end{prop}

\begin{proof}
	Let $\lam_0 \in M$, and let $h_\la: J(f_{\la_0}) \to J(f_\la)$ 
	denote the dynamical holomorphic motion of the Julia set as in Definition~\ref{defi:jstab}.	Let $v \in S(f_{\la_0})$ and $v_\la:=\varphi_\la(v)=h_\la(v)$. 	Let $z$ be a repelling periodic point of $f_{\la_0}$  of period at least 3, with infinitely many preimages 
	and  which is not in the forward orbit of $v$. (Such points always exist by the island property). Assume for a contradiction
	that $v_\la$ is active at  $\la_0$.
	By Proposition \ref{prop:misiu}, there exists $\la_1 \in M$ close to $\la_0$ and $n \in \N$ 
	such that $f_{\la_1}^n(v_{\la_1}) = h_{\la_1}(z)$. {\violet Since $v_{\la_1}=\phi_{\la_1}(v)=h_{\la_1(v)}$ by item \eqref{i:item2} in the definition of $J$-stability,  this gives $f_{\la_1}^n(h_{\la_1}(v)) = h_{\la_1}(z)$. 	But since $h_{\la_1} : J(f_{\la_0}) \to J(f_{\la_1})$ is a topological conjugacy, and  $f^n(v) =f_{\la_0}^n(h_{\la_0}(v)) \neq h_{\la_0}(z)=z$,} this is not possible.	Therefore, $v_\la$ must be passive on $M$.
\end{proof}

\begin{proof}[Proof of Theorem \ref{th:mssahlfors}  ]
	Let $\{f_\la\}_{\la \in M}$ be a natural family of Ahlfors maps.
	Proposition  \ref{prop:passivity} proves that $J$-stability implies passivity of all singular values.
	
	Let us prove that conversely, if $\la_0 \in M$ and if 
	there is a neighborhood $V$ of $\la_0$ such that all critical and asymptotic values are passive on $V$, then there is a neighborhood $U \subset V$ of $\la_0$ such that $\{f_\lam\}_{\la \in U}$ is $J$-stable.

	Let $z_0(\la_0)$ denote a repelling periodic point of period $p \geq 1$
	for $f_{\la_0}$. 
	Let $U \subset V$ denote a simply connected neighborhood $\la_0$ on which 
	$z_0(\la)$ moves holomorphically as a repelling periodic point of $f_\la$.

	By Lemma ~\ref{lem:misiuimpliesactive}, none of the maps $f_\la$, $\la \in U$, may have any non-persistent Misiurewicz relation.  We can therefore apply Corollary \ref{coro:jstable}, which asserts that $\{f_\la\}_{\la \in U}$ is indeed $J$-stable.
\end{proof}

\begin{rem}
	The proof given implies the following: if all critical and asymptotic values are passive, then all singular values are passive. More generally, using Proposition \ref{prop:activity} one could prove directly  that the set of active singular values at a given  parameter $\lam_0 \in M$ is closed.
\end{rem}

\section{Finite type maps and attracting cycles}

We consider, as above,  a natural family $f_\la = \varphi_\la \circ f \circ \psi_\la^{-1}:W_\la \to X$ of finite type maps. In this section we prove some results that are necessary for Theorem \ref{th:MSS}, but which are also of independent interest. 

In the context of \cite{ABF}, where we dealt with natural families of meromorphic maps, i.e. $X=\chat$ and $\partial W_\la=\{\infty\}$ for all $\la\in M$, we were able to prove (see Theorem B, the Accessibility Theorem in \cite{ABF}) that certain active parameter values (those involving asymptotic values mapping eventually to infinity), say $\la_0$,  can always be accessed by a curve of parameters $\la(t)$ such that $f_{\la(t)}$ possesses an attracting cycle, whose multiplier converges to 0 as $\la(t)$ tends to $\la_0$. Despite $\la_0$ being in the bifurcation locus, this property granted them the name of {\em virtual centers}, in a analogy to the centers of hyperbolic components in rational maps. 

The results we prove in this section are the best possible generalization of this property. More precisely, we shall find a {\em sequence} of parameters (instead of a curve) with attracting cycles of the same period and arbitrarily small multiplier (Theorem~\ref{prop:thb}). The difficulty arising in the new contest of finite type maps is that one must account for the possibility of a tract which accumulates in a complicated way on $\partial W$, something which cannot happen for meromorphic maps. For these reason, we shall need a more elaborate definition of what was then called a {\em virtual cycle}.

\subsection{Creation of attracting cycles  near virtual cycles}\label{sec:creation attracting}
%

%

\begin{defi}[Simple virtual cycle]\label{def:vc}
	Let $f: W \to X$ be a finite type map, and let $T$ be a tract above an asymptotic value $v$. We say that $x \in \partial T \cap \partial W$ is a good point in $\partial T$ if it is in the accumulation locus of 
	an oriented hyperbolic geodesic $\gamma \subset T$ which has no limit point in $W$.

	If there is $n \in \N$ such that $f^n(v)$ is a {\em good point} in $\partial T$, we say that $v, f(v), \ldots, f^n(v)$ is a {\em simple virtual cycle} of length $n+1$. 
\end{defi}

Here is an equivalent formulation of the definition above. Let $g : \H \to T$ be a conformal isomorphism between the left half-plane $\H$ and the simply connected tract $T$, normalized so that $f \circ g(z) \to v$ as $\re z \to -\infty$. Then $x \in \partial T \cap \partial W$ is a good point if and only if it there is a constant $y_0 \in \R$ and a sequence $t_k \to +\infty$ such that $g(-t_k+i y_0) \to x$. Moreover, up to choosing an appropriate normalization of $g$, we can assume without loss of generality that $y_0=0$.

\begin{rem}
	In the case of finite type \emph{meromorphic} maps,
	$\partial W=\{\infty\}$ so that the only good point is always $\infty$. Therefore, this definition agrees with the one from \cite{ABF}, in the sense that every simple virtual cycle is a virtual cycle as defined in \cite[Definition 1.3]{ABF} (see also \cite{FK}).
\end{rem}

\begin{defi}[Non-persistency]
Let $v_{\la_0}, f_{\la_0}(v_{\la_0}), \ldots, f_{\la_0}^n(v_{\la_0})$ be a simple virtual cycle for some $\la_0\in M$ and $n\in \N$. We say that the cycle is {\em non persistent}, if $f^n_{\la_0}(v)$ maps to $\partial W_{\la_0}$ non-persistently (see Definition \ref{def:np}).
\end{defi}

 The main theorem in this section is the following.

\begin{theo}[Attracting cycles]\label{prop:thb}
	Assume that there is $\lam_0 \in M$ such that $f_{\lam_0}$ has a non-persistent simple virtual cycle of length $n+1$. Then there is a sequence $\lam_k \to \lam_0$ such that 
	$f_{\lam_k}$ has an attracting cycle of period $n+1$, of multiplier $\rho_k \to 0$, which captures 
	the asymptotic value $v_{\lam_k}$.
\end{theo}

For the proof, we  will use the following   technical lemmas, proved in \cite{ABF}.

\begin{lem}[Hyperbolic distance in tracts { \cite[Lemma 4.1]{ABF}}]\label{lem:lasse} Let $T \subset X$ be a simply connected hyperbolic  domain, $\rho_T$ be the hyperbolic density in $T$ with respect to a continuous hermitian metric on $X$, and let $z,w\in T$. Then 
	\begin{equation}\label{eq:lasse}
		\dist_T(z,w)\geq\frac{1}{2}\left|\ln\frac{\dist(w,\partial T)}{\dist(z,\partial T)}\right|.
	\end{equation}
\end{lem}

\begin{lem}[Asymptotic derivative of the Riemann map { \cite[Lemma 4.2]{ABF}}]\label{lem:nocusp} Let $\H$ be the left half plane,  $T$ be a simply connected hyperbolic  domain, $g:\H\ra T$ be a Riemann map. Then for every $\alpha>0$, 
	\begin{equation}\label{eq:nocusp}
		\lim_{t \to +\infty} |g '(-t)| e^{\alpha t} =\infty. 
	\end{equation}
\end{lem}

\begin{lem}[Distortion of small disks { \cite[Lemma 2.9]{ABF}}]\label{lem:distortion}
	Let $\{\varphi_\lam\}_{\lam \in \D}$ be a holomorphic motion of $X$, with $\varphi_0 = \id$.
	Let $t \mapsto \lam(t)$ be a continuous path in $\D$ with $\lim_{t \to +\infty} \lam(t)=0$, and $t \mapsto r_t$ a continuous function with $r_t >0$ and 
	$\lim_{t \to +\infty} r_t=0$. Let $t \mapsto z_t$ be a path in $X$ and $D_t:=\D(z_t,r_t)$. Let $\epsilon>0$;
	then for all $t$ large enough:
	\begin{equation*}
		\D(\varphi_{\lam(t)}(z_t), r_t^{1+\epsilon}) \subset \varphi_{\lam(t)}(\D(z_t,r_t)) \subset \D(\varphi_{\lam(t)}(z_t), r_t^{1-\epsilon})
	\end{equation*}
\end{lem}

\begin{lem}\label{lem:diamU}
	Let $h_\lam : \D \to \C$ be a holomorphic family of holomorphic maps, with $\lam \in M$ a domain of $\C^m$ containing $0$  and  assume that $h_\lam(0) \equiv 0$ on $M$.
	Let $\lam_k \to 0$ in $M$ and let $\eps_k \to 0$, $\eps_k>0$. Let $U_k$ denote the connected component of $h_{\lam_k}^{-1}(\D(0,\eps_k))$ containing $0$.
	 There exists $c>0$ such that for all $k \in \N$ large enough,
	$$\D(0,c \eps_k) \subset U_k.$$ 
\end{lem}

\begin{proof}
	Since $h_\lam(0)\equiv 0$, there exists a constant $C>0$ such that for all $(\lam,z)$ in a neighborhood of $(0,0) \in M \times \D$, we have
	$$|h_\lam(z)| \leq C |z|.$$
	This means that for all $k$ large enough, $(\lam_k,\eps_k)$ belongs to that neighborhood and
	$$h_{\lam_k}(\D(0,c \eps_k)) \subset \D(0,\eps_k) $$
	where $c:=\frac{1}{C}$. In particular, $\D(0,c \eps_k) \subset U_k$.
\end{proof}

The proof of Theorem \ref{prop:thb} will follow from the next lemma, which we will also need later on.

\begin{lem}[ Finding attracting cycles]\label{lem:thbcore}
	Let $\lam_0 \in M$, and let $v_{\lam_0}$ be an asymptotic value.
	Let $T$ be a tract above $v_{\lam_0}$, and let $\Phi: T \to \H$ be a Riemann uniformization of $T$ onto the left half-plane. Assume that there exists $\lam_k \to \lam_0$ and $t_k \to +\infty$ such that 
	$f_{\lam_k}^{{ n}}(v_{\lam_k}) = \psi_{\lam_k} \circ \Phi^{-1}(-t_k)$.
	Then for all $k$ large enough, $f_{\lam_k}$  has an attracting cycle of period $n+1$, of multiplier $\rho_k \to 0$, which captures 
	the asymptotic value $v_{\lam_k}$.
\end{lem}

\begin{proof}[Proof of Lemma \ref{lem:thbcore}]
	To simplify the notations, set  $f:=f_{\lambda_0}$ and
	$v:=v_{\la_0}$. Recall that $f_\lam=\varphi_\lam \circ f \circ \psi_\lam^{-1}$. Let $V:=\D^*(v, r)$ be a punctured disk centered at $v$ disjoint from $S(f)$, so that
	$f: T\to V$ is a universal cover. In particular, $f(z)=v+r e^{\Phi(z)}$ for all $z \in T$. 
	Let $V_\lam:=\varphi_\lam(V)$ and $T_\lam:=\psi_\lam(T)$, so that 
	$f_\lam:T_\lam\ra V_\lam$ is also  an infinite degree universal cover,
	and let $\Phi_\lam:=\Phi\circ \psi_\lam^{-1} : T_\lam \to \H$.
	Then $\varphi_\lam^{-1} \circ f_\lam : T_\lam \to V$ is a universal cover, and so
	for all $z \in T_\lam$, 
	\begin{equation}\label{eq:expf}
		f_\lam(z)=\varphi_\lam\left(v+ r e^{\Phi_\lam(z)}\right)
	\end{equation}
	By assumption, $f_{\lam_k}^{n}(v_{\lam_k}) = \psi_{\lam_k} \circ \Phi^{-1}(-t_k)$.
	
	Now let $D_{t_k}:=\Phi_{\lam_k}^{-1}(\D(-t_k,\pi))\subset T_{\lam_k}$ and let $U_{t_k}$ denote the connected component of 
	$f_{\lam_k}^{-n}(D_{t_k})$ containing $v_{\lam_k}$. We will prove that for all $k$ large enough, 	$f_{\lam_k}(D_{t_k}) \Subset U_{t_k}$, or equivalently, 
	$f_{\lam_k}^{n+1}(U_{t_k}) \Subset U_{t_k}$; this implies the existence of an attracting fixed point for $f_{\lam_k}^{n+1}$. 
	

	First, let us show that for $r$ small and $k$ large, $f_{\lam_k}(D_{t_k})$ is contained in a small disk centered at $v_{\lam_k}$, or more precisely,   
	\begin{equation}\label{eq:controldt}
		f_{\lam_k}(D_{t_k}) \subset \D\left(v_{\lam_k}, e^{-t_k(1-\epsilon)})\right).
	\end{equation}
	By (\ref{eq:expf}) we have that for all $z \in \H$,
	\[
	f_\lam \circ  \Phi_\lam^{-1}(z)= \varphi_\lam(v +r e^z).
	\]
	Since $\D(-t_k,\pi)\subset\{z\in\C: \Re z< -t_k+\pi\}$ we have that 
	\[
	f_{\lam_k}(D_{t_k}) \subset \varphi_{\lam_k}(\D(v, r e^{-t_k+\pi}))
	\]
	Let $\epsilon>0$. By Lemma \ref{lem:distortion}, we have for all $k$ large enough:
	\begin{equation}
		f_{\lam_k}(D_{t_k}) \subset \D\left(v_{\lam_k}, (r e^\pi)^{1-\epsilon}e^{-t_k(1-\epsilon)})\right),
	\end{equation}
	which for $r$ small implies (\ref{eq:controldt}).
	
	Now we show that $U_{t_k}$ contains a disk centered at $v_{\lam_k}$ whose radius, for $t_k$  large,  is much larger than $e^{-t_k(1-\epsilon)}$.
	
	
	Let us first estimate $\dist(f_{\lam_k}^{n}(v_{\lam_k}), \partial D_{t_k})$. 
	To lighten the notations, let $g:=\Phi^{-1}$; then $g$ is univalent on $\H$
	and $D_{t_k}=\psi_{\lam_k} \circ g(\D(-t_k,\pi))$. By Koebe's theorem, 
	$g(\D(-t_k,\pi))$ contains a disk 
	\[
	\D(g(-t_k),C |g'(-t_k)|).
	\]
	Then, by Lemma \ref{lem:distortion}, 
	\begin{equation}\label{eq:indiamD}
	\begin{split}
			D_{t_k}& =\psi_{\lam_k}\circ g(\D(-t_k,\pi))\supset	 \D(\psi_{\lam_k} \circ g(-t_k),  C^{1+\epsilon}|g'(-t_k)|^{1+\epsilon})\\
		& =	\D( f_{\lam_k}^{n}(v_{\lam_k}), C^{ 1+\epsilon}|g'(-t_k)|^{1+\epsilon}).
	\end{split}
	\end{equation}


	

	 Recall that $U_{t_k}$ denotes the connected component of $f_{\lam_k}^{-n}(D_{t_k})$ containing 
	$v_{\lam_k}$, and let \linebreak $\eps_k:= C^{ 1+\epsilon}|g'(-t_k)|^{1+\epsilon}$.
	Applying Lemma \ref{lem:diamU} to $h_\lam(z):=f_\lam^n(z+v_\lam)-f_\lam^n(v_\lam)$ and using \eqref{eq:indiamD},
	we obtain the existence of $c>0$ such that 
	\begin{equation}\label{eq:controlut}
		\D(v_{\lam_k}, c \eps_k) \subset U_{t_k}.
	\end{equation}

	Finally, from equations \eqref{eq:controldt} and \eqref{eq:controlut}, 
	it is enough to prove that 
	\begin{equation}\label{eq:ratio of radii}
		\frac{e^{-t_k(1-\epsilon)}}{ |g'(-t_k)|^{1+\epsilon}} \to 0 \quad \text{as $t \to +\infty$}, 
	\end{equation}
	which follows from Lemma~\ref{lem:nocusp}. 
	This proves that $f_{\lam_k}^{ n+1}(U_{t_k}) \Subset U_{t_k}$, and the result then follows from Schwartz's lemma. Note that (\ref{eq:ratio of radii}) also implies that  the multiplier   goes to zero
	as $k \to +\infty$,  since the modulus of $U_{t_k}\setminus \overline{f_{\la_k}^{n+1}(U_{t_k})}$ tends to infinity.
\end{proof}

We now finish the proof of Theorem \ref{prop:thb}.

\begin{proof}[Proof of Theorem \ref{prop:thb}]
	By assumption, there exists $t_k \to +\infty$ such that 
	$\Phi^{-1}(-t_k) \to x \in \partial W_{\la_0}$, while $f^n(v)=x$.
		Now, we wish to find a sequence $(\lam_k)_{k \in \N}$ in parameter space such that
		\begin{equation} \label{eq2}
				\Phi_{\lam_k} \circ f_{\lam_k}^{n}(v_{\lam_k}) = -t_k.
			\end{equation}
		
		Let $G(\lam):=\psi_\lam^{-1} \circ f_\lam^{n}(v_\lam)$  and recall that $G(\la_0)=f^n(v)=x$. Given the definition of $\Phi_\lam$, (\ref{eq2}) is equivalent to 
		\begin{equation}\label{eq:curve}
				\Phi\circ \psi_{\lam_k}^{-1} \circ f_{\lam_k}^{n} (v_{\la_k}) = -t_k,
			\end{equation}
		or
		\begin{equation}
				G(\lam_k) = \Phi^{-1}(-t_k).
			\end{equation}
		Since $G$ is a branched cover over a neighborhood of $x$ by Lemma \ref{lem:qr}, there is such a sequence 
		(not necessarily unique if the local degree of $G$ at $\lam_0$ is more than 1).
	We finally apply Lemma \ref{lem:thbcore} to conclude.	
\end{proof}

\subsection{Creation of attracting cycles at active parameters}

We conclude this section with the construction of attracting cycles of high period near parameters with active singular values.
We begin with the easier case of critical values.

\begin{prop}[Super-attracting cycles are dense in the activity locus of critical values]\label{lem:center}
	Let $f_{\la_0}$ be a non exceptional finite type map.
	Let $v_\lam$ be a critical value, and assume that it is active at $\lam_0$.
	Assume as well that there exists a critical point $c(\lam_0) \notin S(f_{\lam_0})$ such that $f_{\lam_0}(c_{\lam_0}) = v_{\lam_0}$. 
	Then there exists $n_k \to +\infty$ and $\lam_k \to \lam_0$ such that $v_{\lam_k}$ is in a super-attracting cycle of period $n_k$.
\end{prop}

\begin{proof}
	By Proposition \ref{prop:trunc}, there exists $n_k \to +\infty$ and $\lam_k \to \lam_0$
	such that $f_{\lam_k}^{n_k}(v_{\lam_k}) \in \partial W_{\lam_k}$. 
	We then apply Proposition \ref{lem:shooting} to find $\lam_k'$ arbitrarily close to 
	$\lam_k$ such that $f_{\lam_k'}^{n_k+1}(v_{\lam_k'}) = c_{\lam_k'}$; in other words,
	$v_{\la_k'}$ is a super-attracting periodic point of period $n_{k}+1$.

\end{proof}

\begin{prop}[Attracting cycles are dense in the activity locus of asymptotic values]  \label{lem:vc}
 Let $f_{\lambda_0}$ be a non exceptional finite type map. 
	Let $v_\lam$ be an  asymptotic value, and assume that it is active at $\lam_0$.
	Then  there exists $n_k \to +\infty$ and $\lam_k \to \lam_0$ such that $f_{\lam_k}$
	has an attracting cycle of period $n_k$.
\end{prop}

\begin{proof}

	Let $\eps>0$ and $N \in \N$: we will find $\lam_* \in \B(\lam_0,\eps)$ with an attracting cycle of period at least $N$.
	By Proposition \ref{prop:trunc}, there exists $n_k \to +\infty$ and $\lam_k \to \lam_0$
	such that $f_{\lam_k}^{n_k}(v_{\lam_k}) \in \partial W_{\lam_k}$ (non-persistently). Let $k \in \N$ such that $n_k \geq N$ and $d( \lam_k , \lam_0) < \frac{\eps}{2}$.
	Let $T$ be a tract above $v_{\lam_0}$, and let $ \Phi: T \to \H$ be a Riemann uniformization onto the left half-plane. Let $t_\ell \to -\infty$ be any sequence, and let $\gamma_\ell(\lam):=\psi_{\lam} \circ \Phi^{-1}(-t_\ell)$. For all $\ell$ large enough, 
	$\gamma_\ell(\lam_0) = \Phi^{-1}(-t_\ell) \notin S(f_{\lam_0})$. By Proposition \ref{lem:shooting}, there exists $ \lam_* \in \B( \lam_k, \frac{\eps}{2} )$ such that $f_{{ \lam_*}}^{{ n_k+2}}(v_{{ \lam_*}})=\gamma_\ell({ \lam_*}) =  \psi_{{ \lam_*}} \circ \Phi^{-1}(-t_\ell)$.
	By Lemma \ref{lem:thbcore}, for all $\ell$ large enough, $f_{\lam_2}$ has an attracting cycle of period $ n_k+3>N$, and we are done. 
	
\end{proof}

\section{ Characterization of $J$-stability: Proof of Theorem \ref{th:MSS}}

In view of Propositions~\ref{lem:exceptional} and \ref{prop:sparsexcep}, in the proof of Theorem~\ref{th:MSS} we can assume that  exceptional maps form a proper analytic subset in the natural family under consideration. Indeed,  affine endomorphisms of the complex torus and automorphisms have no singular values; and for rational maps, entire maps, and meromorphic maps, Theorem~\ref{th:MSS} has been proven in \cite{mss}, \cite{lyu1}, \cite{lyu2}, \cite{EL} and \cite{ABF} respectively.

The case of a natural family of finite type self-maps of $\C^*$ with essential singularities at $0$ and $\infty$ is not formally covered by the aforementionned articles. However, the proof given in \cite{EL} (stated only for finite entire maps) applies verbatim to finite type self-map of $\C^*$. We will therefore only treat the case where the maps $f_\la$ are non exceptional.

\begin{proof}[Proof of Theorem \ref{th:MSS}]
	$ $\\
	
	\begin{itemize}
		\item  $(1) \Leftrightarrow (2)$: This is a particular case of Theorem \ref{th:mssahlfors}, since 
		finite type maps are Ahlfors island maps.

		\item $(2) \Rightarrow (3)$: This part of the proof follows the same argument as in \cite{mss} and \cite{lyu1}.   Assume that the Julia set moves holomorphically over $U$, and let $H_{\la}$ be the holomorphic motion respecting the dynamics as above. Then $H_{\la}$ maps non-attracting periodic points of $f_{\lam_0}$ in $J(f_{\la_0})$ to  non-attracting  periodic points of $f_\lam$ in $J(f_\lam)$ of the same period. Let $N$ be the maximal period of attracting cycles for $f_{\lam_0}$. \footnote{ Note that we only use a very weak form of Fatou-Shishikura's inequality here, namely that a finite type map $f$ has at most $\card S(f)$ attracting cycles (this is an obvious consequence of the fact that any attracting cycle captures at least one singular value). Epstein has an unpublished proof of a strong version of Fatou-Shishikura's inequality in the setting of finite type maps, which we do not require here.   } Then for all $\lam \in U$, cycles of period more than $N$ must be non-attracting, which implies that attracting cycles have period at most $N$.
		\item $(3) \Rightarrow (1)$: Assume by contraposition that at least one singular value $v_\lam$ is active at $\lam_0$, and let $N \in \N$. 
		We must construct a sequence of parameters $\lam_k \to \lam_0$ such that $f_{\lam_k}$ has an attracting cycle of period at least $N$.

		We begin with the following remark: let us say that a singular value $v_1(\lam)$ has a predecessor if there exists another singular value $v_2(\lam)$ and $n >0$ such that 
		for all $\lam \in M$, $f_\lam^n(v_2(\lam)) = v_1(\lam)$. Then clearly $v_1(\lam)$ is active at $\lam_0$ if and only if its predecessor $v_2(\lam)$ is also active at $\lam_0$. Moreover, a singular value is its own predecessor 
		if and only if it is persistently periodic; but in that case it cannnot be active. 
		Therefore, since there are only finitely many singular values, we may assume without loss of generality that $v_\lam$ has no predecessor.

		If $v_\lam$ is an asymptotic value,
		then we are done by Proposition \ref{lem:vc}.

		We therefore assume from now on that $v_\lam$ is a critical value with no predecessor.
			Let $c_{\lam_0}$ be a critical point such that $f_{\lam_0}(c_{\lam_0})=v_{\lam_0}$.
			If $c_{\lam_0}$ is not a singular value, then we are done by Proposition \ref{lem:center}.
			Otherwise, $c_{\lam_0}$ is both a critical point and a singular value, and then 
			$c_\lam:=\psi_\lam(c_{\lam_0})$ is its motion as a critical point and $\varphi_\lam(c_{\lam_0})$
			 is its motion as a singular value. 
			Since by assumption $v_\lam$ has no predecessor, the critical point $c_\lam=\psi_\lam(c_{\lam_0})$ cannot always be a singular value for all $\lam \in M$.
			This means that  $\varphi_\lam(c_{\lam_0}) \not\equiv \psi_\la(c_{\lam_0})$.	
			Then the set $H:=\{\lam : \varphi_\lam(c_{\lam_0}) = \psi_\lam(c_{\lam_0}) \}$ is an analytic hypersurface of $M$, and by Lemma \ref{lem:perturbA}, we may find $\lam_1 \notin H$ arbitrarily close to $\lam_0$ such that $v_{\lam}$ is also active at $\lam_1$. Then $c_{\lam_1}:=\psi_{\lam_1}(c_{\lam_0}) \notin S(f_{\lam_1})$, and $f_{\lam_1}(c_{\lam_1}) = \varphi_{\lam_1}(v_{\lam_0}) = v_{\lam_1}$.
			So we can again apply Proposition \ref{lem:center} and conclude in this case.

	\end{itemize} 
\end{proof}

\begin{proof}[Proof of Corollary \ref{coro:stabdense}]
	The proof follows the same spirit as in \cite{mss}: let $\lam_0 \in M$, and let $\eps>0$.
	If all singular values of $f_{\lam_0}$ are passive at $\lam_0$, then $\lam_0$ is in the stability locus. Otherwise,  at least one attracting singular value is active.
	If it is a critical value, then by Proposition \ref{lem:center} we can find $\lam_1 \in \B(\lam_0,\eps)$ 
	such that $v_{\lam_1}$ is in a super-attracting cycle for $f_{\lam_1}$. In particular, $v_{\lam}$ becomes passive at $\lam_1$.
	If $v_\lam$ is an asymptotic value, then we use  Proposition \ref{lem:vc} and Theorem \ref{prop:thb} instead to find $\lam_1 \in \B(\lam_0,\eps)$ such that $v_{\lam_1}$ is captured by an attracting cycle (and in particular is passive at $\lam_1)$.

	Applying this successively to all active singular values, we find $\lam' \in \B(\lam_0,k\eps)$ (where $k \leq \card S(f_{\lam_0})$) 
	such that all singular values are passive at $\lam'$. By Theorem \ref{th:MSS}, $\lam'$ is then in the stability locus.
\end{proof}

\begin{proof}[Proof of Corollary \ref{coro:robustbif}]	
	We may choose without loss of generality $\la_0$ as the basepoint of our natural family $\{f_\la\}_{\la \in M}$, 
	that is, we set $f:=f_{\la_0}$ and we write $f_\la = \varphi_\la \circ f \circ \psi_\la^{-1}$ for all $\la \in M$.
	The singular value $v(\la)$ is by definition $\varphi_\la(v(\la_0))$.
	
	By Proposition \ref{prop:misiu}, there exists a parameter $\la_1$ arbitrarily close to $\la_0$ such that 
	$v(\la_1)$ is Misiurewicz, i.e. there exists a repelling periodic point $z(\la)$ and $n \in \N$ 
	such that 
	$$f_{\la_1}^n(v(\la_1))=z(\la_1),$$ and this relation is non-persistent, i.e. $ f_\la^n(\phi_\la(v(\la_0)))- z(\la) \not\equiv 0$ on the neighborhood of $\la_1$.

	Since singular values move holomorphically with the parameter, $v(\la_1)$ is also in the interior of $S(f_{\la_1})$.

	 Moreover, we may choose $z(\la)$ so that it is not a Picard exceptional value. 	Up to passing to a covering (see Lemma~\ref{lem:marked}   ) we can consider   $x_\lambda$ such that $f_\lambda^n(x_\lambda)=z(\lambda)$ such that $x_{\lambda_0}=v(\lambda_0)$. Since $v_{\lambda_0}$ was in the interior of $S(f_{\lambda_0})$,  up to restricting the parameter neighborhood we have that $x_\lambda$ is a singular value for $f_{\lambda}$. Hence each such parameter has a Misiurewicz relation which involves the singular value $x(\lambda)$. On the other hand,  for each $\lambda_*$ in a neighborhood,  $\la \mapsto x(\la) - \phi_\la(x(\la_*))$ is not identically zero, hence each such Misiurewicz relation is not persistent and $\la_*$ is in the bifurcation locus, by Lemma \ref{lem:misiuimpliesactive} and Theorem \ref{th:mssahlfors}.

		This proves that $\la_1$ is in the interior of the bifurcation locus, and since $\la_1$ can be taken arbitrarily close to $\la_0$, we have that $\la_0$ is indeed in the closure of the interior of the bifurcation locus.
\end{proof}

\bibliographystyle{alpha}
\bibliography{BifMeroABF}

\end{document}